\documentclass[12pt,reqno]{amsart}
\usepackage[a4paper,left=30mm,right=30mm,top=30mm,bottom=35mm,marginpar=20mm]{geometry}

\usepackage{amsmath,amsthm,amssymb,amsfonts,amsopn}
\usepackage[nobysame]{amsrefs}
\usepackage{url}
\usepackage[dvipsnames]{xcolor}
\usepackage[english=american]{csquotes}
\usepackage{enumerate}
\usepackage{stmaryrd} 
\usepackage{mathrsfs}
\usepackage[normalem]{ulem}
\usepackage{mathtools}
\usepackage{color, colortbl}
\definecolor{Gray}{gray}{0.9}
\usepackage{bbm}
\usepackage{bm}
\usepackage{enumitem}
\usepackage[multiple]{footmisc}
\usepackage{hyperref}
\hypersetup{
    bookmarks=true,         % show bookmarks bar?
    unicode=false,          % non-Latin characters in Acrobats bookmarks
    pdftoolbar=true,        % show Acrobats toolbar?
    pdfmenubar=true,        % show Acrobats menu?
    pdffitwindow=false,     % window fit to page when opened
    pdfstartview={FitH},    % fits the width of the page to the window
    pdftitle={My title},    % title
    pdfauthor={Author},     % author
    pdfsubject={Subject},   % subject of the document
    pdfcreator={Creator},   % creator of the document
    pdfproducer={Producer}, % producer of the document
    pdfkeywords={keyword1, key2, key3}, % list of keywords
    pdfnewwindow=true,      % links in new PDF window
    colorlinks=true,       % false: boxed links; true: colored links
    linkcolor=blue ,          % color of internal links (change box color with linkbordercolor)
    citecolor=blue ,        % color of links to bibliography
    filecolor=magenta,      % color of file links
    urlcolor=Aquamarine           % color of external links
}
\usepackage{caption}
\usepackage{tikz}

\usepackage{oubraces}
\usepackage{siunitx}
\usepackage{xcolor}
\usepackage{booktabs,colortbl, array}
\usepackage{pgfplotstable}
\pgfplotsset{compat=1.8}

\definecolor{rulecolor}{RGB}{0,71,171}
\definecolor{tableheadcolor}{gray}{0.92}

\colorlet{tableheadcolor}{gray!25} % Table header colour = 25% gray
\colorlet{tablerowcolor}{gray!10} % Table row separator colour = 10% gray
\makeatletter
\DeclareFontFamily{OMX}{MnSymbolE}{}
\DeclareSymbolFont{MnLargeSymbols}{OMX}{MnSymbolE}{m}{n}
\SetSymbolFont{MnLargeSymbols}{bold}{OMX}{MnSymbolE}{b}{n}
\DeclareFontShape{OMX}{MnSymbolE}{m}{n}{
    <-6>  MnSymbolE5
    <6-7>  MnSymbolE6
    <7-8>  MnSymbolE7
    <8-9>  MnSymbolE8
    <9-10> MnSymbolE9
    <10-12> MnSymbolE10
    <12->   MnSymbolE12
}{}
\DeclareFontShape{OMX}{MnSymbolE}{b}{n}{
    <-6>  MnSymbolE-Bold5
    <6-7>  MnSymbolE-Bold6
    <7-8>  MnSymbolE-Bold7
    <8-9>  MnSymbolE-Bold8
    <9-10> MnSymbolE-Bold9
    <10-12> MnSymbolE-Bold10
    <12->   MnSymbolE-Bold12
}{}
\let\llangle\@undefined
\let\rrangle\@undefined
\DeclareMathDelimiter{\llangle}{\mathopen}%
{MnLargeSymbols}{'164}{MnLargeSymbols}{'164}
\DeclareMathDelimiter{\rrangle}{\mathclose}%
{MnLargeSymbols}{'171}{MnLargeSymbols}{'171}
\makeatother

\newtheorem{theorem}{Theorem}
\newtheorem{lemma}[theorem]{Lemma}
\newtheorem{proposition}[theorem]{Proposition}
\newtheorem{corollary}{Corollary}[theorem]
\newtheorem{conjecture}[theorem]{Conjecture}

\theoremstyle{definition}
\newtheorem{definition}{Definition}

\theoremstyle{remark}
\newtheorem{remark}{Remark}[theorem]
\newtheorem{example}[theorem]{Example}
\newtheorem{open}{Open Problem}[theorem]

\newcommand{\Crm}{\mathrm{C}}

\newcommand{\Trm}{\mathrm{T}}

\newcommand{\Acal}{\mathcal{A}}
\newcommand{\Bcal}{\mathcal B}
\newcommand{\Ccal}{\mathcal{C}}
\newcommand{\Dcal}{\mathcal{D}}

\newcommand{\Fcal}{\mathcal{F}}
\newcommand{\Gcal}{\mathcal{G}}
\newcommand{\Hcal}{\mathcal{H}}

\newcommand{\Lcal}{\mathcal{L}}
\newcommand{\Mcal}{\mathcal{M}}

\newcommand{\Pcal}{\mathcal{P}}

\newcommand{\Scal}{\mathcal{S}}

\newcommand{\Bfrak}{\mathfrak{B}}

\newcommand{\Abb}{\mathbb{A}}
\renewcommand{\Bbb}{\mathbb{B}}
\newcommand{\Cbb}{\mathbb{C}}

\newcommand{\Fbb}{\mathbb{F}}

\newcommand{\Kbb}{\mathbb{K}}
\newcommand{\Lbb}{\mathbb{L}}

\newcommand{\Nbb}{\mathbb{N}}

\newcommand{\Rbb}{\mathbb{R}}
\newcommand{\Sbb}{\mathbb{S}}

\DeclareMathOperator{\im}{Im}

\DeclareMathOperator{\diverg}{div}
\DeclareMathOperator{\curl}{curl}
\DeclareMathOperator{\dist}{dist}
\DeclareMathOperator{\rank}{rank}

\DeclareMathOperator{\spn}{span}

\DeclareMathOperator{\supp}{supp}

\newcommand{\set}[2]{\left\{\, #1 \  \textup{\textbf{:}}\  #2 \,\right\}}

\newcommand{\setb}[2]{\bigl\{\, #1 \  \textup{\textbf{:}}\  #2 \,\bigr\}}
\newcommand{\setB}[2]{\Bigl\{\, #1 \  \textup{\textbf{:}}\  #2 \,\Bigr\}}
\newcommand{\setBB}[2]{\biggl\{\, #1 \  \textup{\textbf{:}}\  #2 \,\biggr\}}

\newcommand{\dpr}[1]{\langle #1 \rangle}

\newcommand{\cl}[1]{\overline{#1}}

\newcommand{\dd}{\;\mathrm{d}}

\newcommand{\N}{\mathbb{N}}
\newcommand{\R}{\mathbb{R}}
\newcommand{\C}{\mathbb{C}}

\newcommand{\loc}{\mathrm{loc}}
\newcommand{\sym}{\mathrm{sym}}

\newcommand{\tolong}{\longrightarrow}

\newcommand{\toweakstar}{\overset{*}\rightharpoonup}

\newcommand{\todown}{\downarrow}

\newcommand{\embed}{\hookrightarrow}

\newcommand{\SmallO}{\mathrm{\textup{o}}}
\newcommand{\sbullet}{\begin{picture}(1,1)(-0.5,-2)\circle*{2}\end{picture}}
\newcommand{\frarg}{\,\sbullet\,}

\newcommand{\BD}{\mathrm{BD}}

\newcommand{\eps}{\epsilon}

\DeclareMathOperator{\Curl}{Curl}
\DeclareMathOperator{\Tan}{Tan}

\newcommand{\proofstep}[1]{\textit{#1}}

\newcommand{\Leb}{\mathscr L}

\renewcommand{\eps}{\varepsilon}
\renewcommand{\phi}{\varphi}

\newcommand{\M}{\mathcal M}

\makeatletter
\renewcommand*\env@matrix[1][*\c@MaxMatrixCols c]{%
    \hskip -\arraycolsep
    \let\@ifnextchar\new@ifnextchar
    \array{#1}}
\makeatother

\DeclareMathOperator{\Gr}{Gr}

\DeclareMathOperator{\Tr}{Tr}
\DeclareMathOperator{\tr}{tr}

\DeclareMathOperator{\Div}{div}

\newcommand{\mres}{\mathbin{\vrule height 1.6ex depth 0pt width
        0.13ex\vrule height 0.13ex depth 0pt width 1.3ex}}

\def\vint_#1{\mathchoice%
    {\mathop{\kern 0.2em\vrule width 0.6em height 0.69678ex depth -0.58065ex
            \kern -0.8em \intop}\nolimits_{\kern -0.4em#1}}%
    {\mathop{\kern 0.1em\vrule width 0.5em height 0.69678ex depth -0.60387ex
            \kern -0.6em \intop}\nolimits_{#1}}%
    {\mathop{\kern 0.1em\vrule width 0.5em height 0.69678ex depth -0.60387ex
            \kern -0.6em \intop}\nolimits_{#1}}%
    {\mathop{\kern 0.1em\vrule width 0.5em height 0.69678ex depth -0.60387ex
            \kern -0.6em \intop}\nolimits_{#1}}}

\newcommand{\aveint}[2]{\mathchoice%
    {\mathop{\kern 0.2em\vrule width 0.6em height 0.69678ex depth -0.58065ex
            \kern -0.8em \intop}\nolimits_{\kern -0.45em#1}^{#2}}%
    {\mathop{\kern 0.1em\vrule width 0.5em height 0.69678ex depth -0.60387ex
            \kern -0.6em \intop}\nolimits_{#1}^{#2}}%
    {\mathop{\kern 0.1em\vrule width 0.5em height 0.69678ex depth -0.60387ex
            \kern -0.6em \intop}\nolimits_{#1}^{#2}}%
    {\mathop{\kern 0.1em\vrule width 0.5em height 0.69678ex depth -0.60387ex
            \kern -0.6em \intop}\nolimits_{#1}^{#2}}}

\title[]{A look into some of the fine properties of functions with bounded $\Acal$-variation}

\date{\today}

\author[A. Arroyo-Rabasa]{Adolfo Arroyo-Rabasa}
\address[A. Arroyo-Rabasa]{Dipartimento di Matematica,
Largo Bruno Pontecorvo 5, 56127 Pisa, Italy}
\email{\href{mailto:adolfo.rabasa@unipi.it}{adolfo.rabasa@unipi.it}}

\author[A. Skorobogatova]{Anna Skorobogatova}
\address[A. Skorobogatova]{
Department of Mathematics, ETH Zürich, Rämistrasse 101, 8092 Zürich, Switzerland
}
\email{\href{mailto:Anna.skorobogatova@math.ethz.ch}{Anna.skorobogatova@math.ethz.chu}}

\thanks{This project has received funding from the European Research Council (ERC) under the grant No 757254 (SINGULARITY) and the grant No 101078057 (ConFine).}

\begin{document}
    
    \begin{abstract}
        We establish certain fine properties for functions of bounded $\mathcal A$-variation known in the classical $BV$ setting. Here, $\mathcal A$ is a $k$th order constant-coefficient homogeneous linear differential operator with a finite-dimensional kernel (also known as a complex-elliptic operator).
        We prove that if $\Acal u$ can be represented by a finite Radon measure, then the potential $u$ has one-sided $L^p$-approximate limits on Lipschitz hypersurfaces, and, more generally, on countably rectifiable sets of codimension one. We use this to give pointwise characterizations of the (functional) interior and exterior traces. We also establish a quantitative scale-dependent continuity result, which allows us to prove that the Lebesgue discontinuity set has zero $(n-1)$-dimensional Riesz capacity. Lastly, we introduce a decomposition that reduces the complexity of analyzing $k$th-order operators to that of first-order methods and allows us to establish the $k$th order $L^p$-differentiability of $BV^\Acal$ maps.
    \end{abstract}

    \maketitle
\setcounter{tocdepth}{1}
{
  \hypersetup{linkcolor=black}
\tableofcontents
}

    \section{Introduction}
    The space of functions of bounded variation $BV(\Omega;\R^M)$ on an open set $\Omega \subset \R^n$ and with values on $\R^M$ consists of all functions $u \in L^1(\Omega;\R^M)$ for which the distributional gradient  can be represented by a matrix-valued finite Radon measure $Du \in \M(\Omega;\R^M \otimes \R^n)$. This space has been studied in great depth by numerous authors, resulting in an extensive classification of the various properties, most of which can be found in \cites{AFP2000,evans1992} and references therein. 
 \textsc{De Giorgi} \cites{de-giorgi-su-una-teoria-1954,de-giorgi-nuovi-teoremi-1955,de-giorgi-sulla-proprieta-1958,de-giorgi1961frontiere-orien}, and \textsc{Federer} \cites{federer-note-on-gauss-green-1958,federer1969geometric-measu} studied the particular class of BV functions that consists of characteristic functions of \emph{sets of finite perimeter}, and in 1960, \textsc{Fleming  \& Rishel} \cite{Fleming_Rishel_60} proved the \emph{co-area formula} 
    \[
    |Du|(B) = \int_{-\infty}^\infty |D\mathbf{1}_{\{u > t\}}|(B)\dd t.
    \]
   The existence of such a decomposition into a family of one-dimensional sections is an example of the \textit{fine properties} for BV functions. 
 These were studied in more detail by \textsc{Federer} \cite{federer1969geometric-measu} and \textsc{Vol'pert} \cite{Vol_pert_1967} 
    in the 1960s. Almost simultaneously, but using different methods, they showed that for $u \in BV(\Omega;\R^M)$ one can decompose the total variation gradient measure into mutually singular measures as
    
    \begin{equation}\label{eq:dec}
    Du = \nabla u \, \Leb^n \, + \, (u^+ - u^-) \otimes \nu_u \, \Hcal^{n-1} \mres J_u \, \, + D^c u.
    \end{equation}
    Here, $\nabla u$ denotes the density of the absolutely continuous part of $D u$ with respect to the $n$-dimensional Lebesgue measure $\Leb^n$,  $\Hcal^{n-1}$ denotes the $(n-1)$-dimensional Hausdorff measure, the \emph{jump set} $J_u$ is the set of \emph{approximate discontinuity points} $x$ where $u$ has one-sided limits $u^+(x) \neq u^-(x)$ with respect to a suitable orientation $\nu_u(x) \in \Sbb^{n-1}$, and the \emph{Cantor part} $D^c u$ is the restriction of the singular part (with respect to $\mathscr{L}^n$) $D^s u$ of $D u$ to the set where $u$ is \emph{approximately continuous}:
       \begin{definition}[Approximate continuity]\label{def:appcont}
    	Let $u \in L^1_\loc(\Omega;\R^M)$ and let $x \in \Omega$. We say that $u$ has an \emph{approximate limit} $z \in \R^M$ at $x$ if
    	\[
    	\lim_{r \todown 0} \aveint{B_r(x)}{} |u(y) - z| \dd y = 0.
    	\]
    	The set of points $S_u \subset \Omega$ where this property fails is called the \emph{approximate discontinuity set}. Notice that $J_u \subset S_u$, and Lebesgue's theorem asserts that $\mathscr L^n(S_u) = 0$.
    \end{definition}

Recall that a related decomposition to~\eqref{eq:dec} holds for  $\BD(\Omega)$, which is the space of functions of \emph{bounded deformation}, consisting of all maps $u \in L^1(\Omega;\R^n)$ whose distributional symmetric gradient
    \[
    Eu = \frac 12 \left(\frac{\partial u^j}{\partial x_i} + \frac{\partial u^i}{\partial x_j}\right)_{i,j = 1,\dots,n} 
    \]
    can be represented by a finite Radon measure in $\M(\Omega,(\R^n \otimes \R^n)_\sym)$. 
    Already for the $\BD$-theory there is a drawback in the sense that the \emph{size} of $S_u$ has not yet been fully understood. The sharpest result in this context is due to \textsc{Kohn} \cite{kohn1979new-estimates-f}*{Part~II}, who obtained the capacity estimate
    \begin{equation*}\label{eq:cap_est}
    \mathrm{Cap}_{n-1}(S_u \setminus J_u) = 0,
    \end{equation*}
    where $\mathrm{Cap}_{n-1}$ denotes the Riesz $(n-1)$-capacitary measure (for a definition, see Section~\ref{sec:cap}). This, in particular, implies the dimensional bound    
    $\Hcal^{n-1+\eps}(S_u\setminus J_u) = 0$ for every $\eps > 0$. 
    Concerning the full picture for the properties of functions of bounded deformation, \textsc{Ambrosio, Coscia \& Dal Maso}~\cite{ambrosio1997fine-properties}  further showed that $|Eu|$-almost every point is either an approximate continuity point or an approximate jump point.\\

    The purpose of this work is to extend the classical \emph{fine properties} of $BV$-theory to spaces of functions of bounded $\Acal$-variation, where $\Acal$ is a $k$\textsuperscript{th} order homogeneous linear partial differential operator with constant coefficients between  finite-dimensional (inner product) euclidean spaces $V$ and $W$, of respective dimensions $M$ and $N$. More precisely, we shall consider operators  acting on smooth maps $u \in \Crm^\infty(\Omega;V)$ as
    \begin{equation}\label{eq:B}
    \Acal u \, = \, \sum_{|\alpha| = k} A_\alpha \partial^\alpha u \, \in \, \Crm^\infty(\Omega;W),
    \end{equation}
    where the coefficients $A_\alpha \in W \otimes V^* \cong \mathrm{Lin}(V;W)$ are assumed to be constant. Here, $\alpha \in \N^n_0$ is a multi-index with modulus $|\alpha| = \alpha_1 + \dots + \alpha_n$, and $\partial^\alpha$ denotes the distributional derivative $\partial_1^{\alpha_1}\cdots \partial_n^{\alpha_n}$. Our main structural assumption on $\Acal$ will be that it is a (real-)\emph{elliptic} or \emph{complex-elliptic} operator:\footnote{The nomenclature ``$\Cbb$-elliptic'' or ``complex-elliptic'' was coined in~\cite{breit2017traces}. The concept, however, was introduced by K.T. Smith 
    in~\cites{smith1,smith2} without specific terminology. Smith's definition requires that the complex rank of the tensor $\Abb^k(\xi)$ equals the dimension of its domain  (or equivalently, that $\Abb^k(\xi)$ is injective) for all non-zero frequencies $\xi \in \C^n$. This is equivalent to condition~\eqref{eq:C}, as can be easily shown}. 

    \begin{definition}[$\R$-elliptic and $\C$-elliptic]\label{d:elliptic}
    Let $\Kbb \in \{\Rbb,\Cbb\}$. An operator as above is called $\Kbb$-elliptic when the (principal) symbol map
    \[
    \mathbb \Abb^k(\xi) \, \coloneqq \sum_{|\alpha| = k} \xi^\alpha A_\alpha \, \in \, \mathrm{Lin}(\Kbb \otimes V; \Kbb \otimes W), \quad \xi \in \Kbb^n,
    \]
    is an injective linear map.\footnote{When $k = 1$, we shall simply denote the principal symbol by $\Abb$.} Namely, there exists a positive constant $c$ such that
    \begin{equation}\label{eq:C}
    |\Abb^k(\xi)v| \ge c |\xi|^k|v| \qquad \text{for all $\xi \in \Kbb^n$ and all $v \in \Kbb \otimes V$.}
    \end{equation}
    \end{definition}

    Our goal is to analyze the structure of functions whose distributional $\Acal$-gradient can be represented by a finite Radon measure $\Acal u \in \M(\Omega;W)$. 
    \begin{definition}
    The space $BV^\Acal(\Omega)$ of \emph{functions with bounded $\Acal$-variation on $\Omega$} is defined as
    \[
    BV^\Acal(\Omega) \coloneqq \set{u \in L^1(\Omega;V)}{\Acal u \in \Mcal(\Omega;W)}\,.
    \]
    We endow this space with the norm 
    \[
    \|u\|_{BV^\Acal(\Omega)} \coloneqq \|u\|_{L^1(\Omega)} + |\Acal u|(\Omega)\,,
    \] 
    which gives it a Banach space structure.  
    \end{definition}
    The space $BV^\Acal$ and its properties (under various conditions on $\Acal$) have been studied extensively in recent years; see, for instance~\cites{arroyo2020slicing,breit2017traces,Raita_critical_Lp,RS20,VS2}. The problem of translating the classic theory for gradients to other elliptic operators is highly non-trivial, as first observed by \textsc{Ornstein}~\cite{Ornstein}, who pointed out (in a slightly different language) that the study of the fine properties of $BV^\Acal$-functions cannot be reduced to the study of the fine properties of $BV$-functions (see also~\cite{CFM05} and~\cite{KK11}).

   Similar to~\eqref{eq:dec}, for $u \in BV^\Acal(\Omega)$, we can decompose $\Acal u$ into its absolutely continuous, Cantor, and jump parts:
     \begin{align*}
    \Acal u \, & = \,\Acal^{a} u \, + \, \Acal^s u \\
    & = \, \Acal^{a} u \, + \, \Acal^s u \mres (\Omega \setminus J_u) \, + \, \Acal^s u \mres J_u \\
    & \eqqcolon \, \Acal^a u   \, + \, \Acal^c u \, + \, \Acal^j u \,.
    \end{align*}
    The theory of $L^p$-differentiability for convolution products developed by \textsc{Alberti, Bianchini and Crippa} in~\cite{Alberti_Bianchini_Crippa_14} (see also preceding works~\cites{ambrosio1997fine-properties,APR_Critical,hajlasz_on_appr_diff_bd_96}) guarantees the existence of the $L^p$-derivative $\nabla u$ of a function $u \in BV^{\Acal}$, for a general first-order elliptic operator $\Acal$. Based on this result, \textsc{Gmeineder and  Rai{\c{t}}{\u{a}}} established in~\cite{GR17}*{Lemma~3.1} the almost everywhere approximate differentiability for all such functions $u$. There, it is also shown that
    $\Acal^a u = A(\nabla u) \, \Leb^n$, where $A$ is the linear map that expresses $\Acal u$ in jet form (i.e., $\Acal u = A(Du)$).  The identification of the density of $\Acal^a u$ with $A(\nabla^k u)$ for elliptic operators of order $k>1$, as well as the establishment of higher-order $L^p$-differentiability, missing in the literature, are covered in Section~\ref{sec:higher_order}.
    It is important to note that when $\Acal = D$ (the gradient operator), the Federer-Vol'pert theorem tells us that the Cantor part $\Acal^c u$ is precisely given by $\Acal^s u \mres (\Omega \setminus S_u)$. Similar results hold for the symmetric gradient (\cite{ambrosio1997fine-properties}) and other elliptic operators satisfying one-dimensional slicing representations (see~\cite{arroyo2020slicing}).

         \section{Main results}\label{sec:main} 

   Let us now begin the exposition of our results. We provide only the statements in this section; the proofs appear in the sections thereafter. We characterize the structure of $u \in BV^\Acal(\Omega)$ under assumptions of varying strength in order to determine the sharpness of our results. In this regard, and for the purpose of simplicity, we shall state our main results only for \emph{first order} operators. However, the results presented below also hold for operators of arbitrary order $k \in \Nbb$, with $u$ replaced by its $(k-1)$-st order gradient $\nabla^{k-1} u$, the jump set $J_u$ replaced by $J_{\nabla^{k-1}u}$, and other minor modifications (see Appendix~\ref{sec:h} below for details). 
   
   For our statements, it will be essential to define (following~\cite{kohn1979new-estimates-f}) the set
    \[
    \Theta_u \coloneqq \setBB{x \in \Omega}{\limsup_{r \todown 0} \frac{|\Acal u|(B_r(x))}{r^{n-1}} > 0},
    \] 
    of points with positive \emph{Hausdorff $(n-1)$-dimensional upper density}.
    	We also recall that a Borel set $\Gamma \subset \R^n$ is called  \emph{countably $\Hcal^{m}$-rectifiable} if $\Gamma$ can be covered
    by countably many $(n-1)$-dimensional Lipschitz graphs.

\subsection{The jump part for elliptic operators}\label{ss:jump} We begin with the hypothesis of merely ellipticity. Let us first recall the definition of the \emph{approximate jump}:
\begin{definition}[Approximate jump\label{def:jump}] Let $u \in L^1_\loc(\Omega;\R^M)$.
	We say that a point $x$ is an \emph{approximate jump point} of $u$ (henceforth written as $x \in J_u$) if there exist \emph{distinct} vectors $a,b \in \R^M$ and a direction $\nu \in \Sbb^{n-1}$ satisfying 
	\begin{equation}\label{eq:jumps}
		\begin{cases} \displaystyle
			\lim_{r \todown 0} \aveint{B^+_r(x,\nu)}{} |u(y) - a| \dd y = 0,\\[14pt]
			\displaystyle
			\lim_{r \todown 0} \aveint{B^-_r(x,\nu)}{} |u(y) - b| \dd y = 0.
		\end{cases}
	\end{equation}
	Here, we use the notation
	\[
	B^+_r(x,\nu) \coloneqq \setB{y \in B_r(x)}{\dpr{\nu,y} > 0},\quad B^-_r(x,\nu) \coloneqq \setB{y \in B_r(x)}{\dpr{\nu,y} < 0},
	\]
	for the $\nu$\emph{-oriented half-balls} centred at $x$, where $B_r(x)$ is the open unit ball of radius $r > 0$ and centered at $x$. 
	We refer to $a,b$ as the \emph{(approximate) one-sided limits} of $u$ at $x$ with respect to the orientation $\nu$. 
	Since the jump triplet $(a,b,\nu)$ is well-defined up to a sign in $\nu$ and a permutation of $(a,b)$, we shall write $(u^+,u^-,\nu_u) : J_u \to \R^M \times \R^M \times \Sbb^{n-1}$ to denote the triplet Borel map associated to the jump discontinuities on $J_u$, i.e., 
	\[
	x \in J_u \quad \Leftrightarrow \quad \text{\eqref{eq:jumps} holds with $(a,b,\nu) = (u^+(x),u^-(x),\nu_u(x))$}.
	\]
\end{definition}
We are able to show the following:
   
    \begin{theorem}[Jump part characterization I]\label{lem:jump} Let $\Acal$ be a first-order elliptic operator and let $u\in BV^\Acal(\Omega)$. Then, $J_u \subset \Theta_u$ and there exists a countably $\Hcal^{n-1}$-rectifiable set $G_u \subset J_u$ (with orientation $\nu_u$) satisfying $$\Hcal^{n-1}(J_u \setminus G_u) = 0$$ and such that
        \[
        \Acal^s u \mres G_u = \Abb(\nu_u)[u^+ - u^-]\, \Hcal^{n-1}\mres G_u.
        \]
    \end{theorem}
    \begin{remark}\label{rem:Giacomo}
    Following the publication of the initial version of this paper, \textsc{Del Nin}   proved in~\cite{del2020rectifiability} the remarkable fact that for $u \in L^1_\loc(\Omega)$, the jump set $J_u$ is itself countably $\Hcal^{n-1}$-rectifiable. 
    \end{remark}
    Motivated by this result and the countable rectifiability of the jump set for locally integrable functions, we ask the following question:
    \begin{open}\label{op:Nin}
      Does $|\Acal u|(J_u \setminus G_u) = 0$ for all $u \in BV^\Acal(\Omega)$ and all first-order elliptic operators $\Acal$?
    \end{open}
    
  \subsection{Characterization of traces for complex elliptic operators}\label{ss:traces}
   The notion of complex ellipticity originated with the work of \textsc{Aronszajn}~\cite{aron} and \textsc{Smith}~\cites{smith1,smith2}. Smith established it as a \emph{necessary and sufficient} condition for the validity of the coercive $L^p$-inequality 
  \[
  	\|u\|_{W^{1,p}(\Omega)} \lesssim \|u\|_{L^p(\Omega)} + \| \Acal u \|_{L^p(\Omega)} 
  \]
  on Lipschitz domains and for exponents $p$ in the range $(1,\infty)$. Smith demonstrated that complex ellipticity is equivalent to requiring the distributional kernel of the operator to be a finite-dimensional subspace of polynomials. This simple characterization allows one to verify that the gradient, the Hessian, and the symmetric gradient all belong to the class of complex elliptic operators. On the other hand, very well-studied operators, such as the Cauchy-Riemann equations or the Laplacian, are not complex-elliptic.    Recently, \textsc{Breit, Diening} and \textsc{Gmeineder}~\cite{breit2017traces}  have shown that when $\Acal$ is a first-order operator, complex-ellipticity is also \emph{a sufficient and necessary} condition for the existence of an exterior trace operator in $BV^\Acal(\Omega)$. Lastly, let us recall that \textsc{Gmeineider} and \textsc{Raita} ~\cite{GR17} showed that, for first-order operators, complex-ellipticity of $\Acal$ implies the critical embedding $BV^\Acal(\Omega) \embed L^{(n-1)/n}(\Omega;V)$. 
While the fine properties for functions lying in $BV$ and $BD$ spaces are by now well-understood, the general picture concerning the fine properties of $BV^\Acal$-functions is rather incomplete. This article provides a step towards gaining a better insight into some of these properties. 
  
 To begin the exposition of our results, let us recall the background theory concerning the existence of functional traces established in~\cite{breit2017traces}. There, the authors showed that all $n$-dimensional Borel subsets $U \subset \Omega$ with Lipschitz boundary $\partial U$ possess a continuous (exterior) linear trace operator $\tr_U : BV^\Acal(U) \to L^1(\partial U;V)$ satisfying 
    \begin{equation}\label{eq:extension}
    \tr_U(u) = u|_{\partial U}  \qquad \forall u \in BV^\Acal(U) \cap \Crm(\cl U;V),
    \end{equation}
    which in turn gives the characterization
    \begin{equation}\label{eq:GG}
    \Acal u \mres \partial U = \Abb(\nu)[\tr_{U}(u) - \tr_{\Omega \setminus U}(u)] \, \Hcal^{n-1} \mres \partial U,
    \end{equation}
    where $\nu : \partial U \to \Sbb^{n-1}$ is the outer normal to $\partial U$ (defined $\Hcal^{n-1}$-a.e.), and $\Abb$ is the symbol of the operator $\Acal$ as in Definition \ref{d:elliptic} (with $k=1$).
    
	Our main contribution to understanding the functional trace operator hinges heavily on the following result, where we establish that all functions of bounded $\Acal$-variation possess one-sided $L^{{n}/{(n-1)}}$-approximate limits when restricted to countably rectifiable sets: 
    
    \begin{theorem}[One-sided limits on interior rectifiable sets]\label{thm:twoside} Let $\Acal$ be a first-order complex-elliptic operator, let $u$ be a function in $ BV^\Acal(\Omega)$, and let $\Gamma \subset \Omega$ be a countably $\Hcal^{n-1}$-rectifiable set oriented by $\nu$. There exist Borel maps $u_\Gamma^+, u_\Gamma^- : \Gamma \to V$ satisfying 
               \begin{equation} \label{eq:6}
        \begin{cases} \displaystyle
        \lim_{r \todown 0} \aveint{B^+_r(x,\nu(x))}{} |u - u_\Gamma^+(x)|^\frac{n}{n-1} \dd y = 0\\[14pt]
        \displaystyle
        \lim_{r \todown 0} \aveint{B^-_r(x,\nu(x))}{} |u - u_\Gamma^-(x)|^\frac{n}{n-1} \dd y = 0\\[14pt]
        \end{cases}
              \end{equation}
	for $\Hcal^{n-1}$-a.e. $x$ in $\Gamma$. 
	Moreover, 
	\[
		\|u^+_\Gamma - u^-_\Gamma\|_{L^1(\Gamma)} \le C |\Acal u|(\Omega)
	\] 
	for some constant $C$ depending only on $\Abb$.
        \end{theorem} 
    
In summary, there is only one type of discontinuity of $u$ when restricted to any countably $\Hcal^{n-1}$-rectifiable set $\Gamma$. Namely, either $u$ is $L^{n/(n-1)}$-approximately continuous (this is the case $u^+_\Gamma = u^-_\Gamma$) or $u$ has an $L^{n/(n-1)}$ jump-type discontinuity across the orientation of $\Gamma$. 
   Theorem \ref{thm:twoside} implies that the exterior trace operator $\tr_U$ from~\cite{breit2017traces} can be pointwise characterized by the pointwise \emph{measure-theoretic trace}, thus extending the representation~\eqref{eq:extension} to all functions of bounded $\Acal$-variation:
        \begin{corollary}[Characterization of the exterior trace]\label{exterior}
		Let $\Omega \subset \R^n$ be a bounded open set with Lipschitz boundary and let $u \in BV^\Acal(\Omega)$. Then, for $\Hcal^{n-1}$-a.e. $x \in \partial \Omega$
		\[
			\lim_{\rho \todown 0} \aveint{B_r(x) \cap \Omega}{} |u(y) - \tr_\Omega(x)|^{\frac{n}{n-1}} \, dy = 0\,.
		\]
	\end{corollary}

        Another direct consequence of the existence of approximate one-sided limits is the following characterization of the precise representative:
        \[
        u^\star = \frac{u^+_\Gamma + u^-_\Gamma}{2} \qquad \text{$\Hcal^{n-1}$ a.e. on $\Gamma$}\,,
        \]
	where we recall that the precise representative $u^\star$ of $u$ is the Borel map defined as
	    	\[
	    	u^\star(x) \coloneqq 
	    	\begin{cases} \displaystyle
		    		\lim_{r \todown 0} \aveint{B_r(x)}{} u(y)  \dd y & \text{if the limit exists\,,}\\[12pt]
		    		\displaystyle
		    		0 &\text{else\,,}
		    	\end{cases}
	    	\]
which, by Lebesgue's continuity theorem, is integrable and belongs to the same $L^1(\Omega)$-equivalence class as $u$.
	This further conveys an \textbf{explicit representation} of the existing (see~\cite{VS2}*{\S6}) continuous interior trace operator on Lipschitz hypersurfaces:
        \begin{corollary}[Interior trace]\label{interior} Let $\Omega \subset \R^n$ be an open set and let $M\subset \Omega$ be a Lipschitz hypersurface. Then, the Borel map
        \[
			\Tr_M u(x) \coloneqq  u^\star(x) \,, \qquad x \in M\,,
        \]
        defines surjective linear operator $\Tr_M : BV^\Acal(\Omega) \longrightarrow L^1(M;V)$ satisfying
        \[
        		\|\Tr_M u\|_{L^1(M)} \le C\|u\|_{BV^\Acal(\Omega)}
        \]
        for some constant depending on $\Abb$ and $M$. 
        \end{corollary}
        \begin{remark} While Theorem~\ref{thm:twoside} establishes summability of $u^+_\Gamma - u^+_\Gamma$ on arbitrary rectifiable sets $\Gamma$, the current corollary requires Lipschitz surfaces. This is a necessary restriction: the function $u^\star$ may not be integrable on general countably $\Hcal^{n-1}$-rectifiable sets, even in dimension one.
\end{remark}

\subsection{Characterization of the jump part for complex-elliptic operators}\label{ss:jump-C} Thanks to the characterization of the one-sided traces on rectifiable sets, we are able to fully characterize the density of $\Acal u$ when restricted to countably rectifiable sets, and, in particular, the jump part  $\Acal^j u = \Acal^s u\mres J_u$. The precise statement is the following:
	
 \begin{corollary}[Jump part II]\label{thm:jump2} Let $\Acal$ be a first-order complex-elliptic operator, let $u$ be a function in $ BV^\Acal(\Omega)$, and let $\Gamma \subset \Omega$ be a countably $\Hcal^{n-1}$-rectifiable set oriented by $\nu$. Then,
\begin{enumerate}[label={\arabic*.}]
\item The restriction of $\Acal u$ to $\Gamma$ is the $\Hcal^{n-1}$-rectifiable measure 
	\[
        \Acal u \mres \Gamma \, = \, \Abb(\nu)[u^+_\Gamma - u^-_\Gamma] \, \Hcal^{n-1} \mres (\Gamma \cap J_u).
        \]
        In particular, 
        \[
        (u^+_\Gamma,u^-_\Gamma,\nu) = (u^+,u^-,\nu_u) \quad \text{$\Hcal^{n-1}$ almost everywhere on $\Gamma \cap J_u$\,.}
        \]
        
        \item  The jump part of $\Acal u$ is an $\Hcal^{n-1}$ rectifiable measure, and it is given by
         \[
          \Acal^j u = \, \Abb(\nu_u)[u^+ - u^-] \, \Hcal^{n-1} \mres J_u.
        \]
        \item The set $S_u \setminus J_u$ is $\Hcal^{n-1}$ purely unrectifiable, i.e., 
        \[
        		\Hcal^{n-1}(M \cap (S_u \setminus J_u)) = 0 
        \] 
        for all Lipschitz hypersurfaces $M \subset \Omega$.
        \end{enumerate}
\end{corollary}

   \subsection{Approximate continuity properties}\label{ss:app-cty} Lastly, we discuss the \enquote{size} of the set where a $BV^\Acal$-function is approximately discontinuous but does not have a jump discontinuity.

    In the classical $BV$-theory, every point $x \notin J_u$ is an approximately continuous point, except for an $\Hcal^{n-1}$-negligible set. The proof, however, hinges heavily on using the co-area formula for classical BV functions and the theory of sets of finite perimeter. Already for the space $\BD$, where strong slicing techniques exist, this property remains uncertain and belongs to a longstanding conjecture regarding the size of the set $S_u \setminus \Theta_u$. More precisely, it has been conjectured that
    \[
    S_u \setminus \Theta_u \quad \text{is $\sigma$-finite with respect to $\Hcal^{n-1}$.}
    \] 
    \textsc{Kohn} showed in his Ph.D. 
    Thesis~\cite{kohn1979new-estimates-f}*{Part~II and~Thm.~5.15} that this is a critical result, in the sense that $\mathrm{Cap}_{n-1}(S_u \setminus\Theta_u) = 0$. On the other hand, using the strong symmetries of the symmetric gradient operator, \textsc{Ambrosio et al.}~\cite{ambrosio1997fine-properties} have shown through a slicing argument that the set $S_u \setminus\Theta_u$ is $|Eu|$-negligible. Recently, in~\cite{arroyo2020slicing}, the first author has shown that $|\Acal u|(S_u \setminus \Theta_u) = 0$ for all elliptic operators satisfying one-dimensional slicing representations. This, however, is a strict subclass of the class of complex-elliptic operators. 
    
    To commence the discussion about the continuity properties of $BV^\Acal$-functions, we begin by recalling the following Poincar\'e inequality for complex-elliptic operators (established in~\cite{GR17}; see Section~\ref{sec:properties} for more details): There exists a positive integer $\ell = \ell(\Acal)$, a $V$-valued polynomial of $p_{x,r}u$ degree at most $\ell - 1$, and a positive constant $c = c(n,\Acal)$ such that         \begin{equation}\label{eq:TSineq} 
        \|u - p_{x,r}u\|_{L^{n/(n-1)}(B_r(x))} \le c \, |\Acal u|(B_r(x)),
        \quad \text{whenever $B_r(x) \subset \Omega$.}
        \end{equation}
We also introduce the short-hand notation
    \[
        		\theta^{*(n-1)}(\Acal u,x) \coloneqq \limsup_{r \todown 0} \frac{|\Acal u|(B_r(x))}{r^{n-1}}\,.
    \]
for the pointwise upper $(n-1)$-dimensional Hausdorff density of $\Acal u$.
Based on the Poincar\'e estimate~\eqref{eq:TSineq} and an idea of \textsc{Kohn}, we give a remarkably simple proof of following quantitative scale-dependent continuity for complex-elliptic operators:

    \begin{theorem}[Quantitative approximate continuity]\label{prop:capacity} Let $\Acal$ be a first-order complex-elliptic operator, let $u$ be a function in $BV^\Acal(\Omega)$, and write 
        \[
        p_{x,r}u(y) = \sum_{\substack{\beta \in \Nbb^n_0\\|\beta| \le \ell - 1}} a_\beta(x,r) \ y^\beta, \qquad a_\beta(x,r) \in V.
        \]
        to denote the coefficient decomposition of $p_{x,r} u$, the polynomials from~\eqref{eq:TSineq}.
       Then,         \begin{equation}\label{eq:quasihot}
        \limsup_{r \todown 0} \, r^{|\beta|} \ |a_\beta(x,r)| \lesssim\theta^{*(n-1)}(\Acal u,x) \qquad \text{ for all $1 \le |\beta|\le \ell - 1$} 
        \end{equation}
 	and 
        \begin{equation}\label{eq:quasicty}
        \limsup_{r \todown 0} \left\{\inf_{z \in V}  \aveint{B_r(x)}{} |u(y) - z|^\frac{n}{n-1}  \dd y \right\} \lesssim \theta^{*(n-1)}(\Acal u,x).
        \end{equation}
               In particular, the limits in~\eqref{eq:quasihot}-\eqref{eq:quasicty} exist and equal zero for any point $x \in \Theta_u^\complement$.
    \end{theorem}

    The quantitative \emph{scale-dependent} continuity~\eqref{eq:quasicty} will serve as a stepping-stone towards the following capacitary estimate on the approximate discontinuity set; see also~\cite{diening2019continuity}, where a slightly weaker version ($L^1$-approximate continuity) was established using a different approach):
    \begin{corollary}\label{thm:cap_est}  Let $\Acal$ be a first-order order complex-elliptic operator and assume that $x \in \Omega$ is a point for which
        \[
        \int_{B_r(x)} \frac{ \dd |\Acal u|(y)}{|x - y|^{n-1}}  < \infty \qquad \text{for some $0 < r <  \dist(x,\partial\Omega)$}.
        \]    
        Then $u$ is $L^{n/(n-1)}$-approximately continuous at $x$, that is,
        \[
        \limsup_{r \todown 0}  \aveint{B_r(x)}{} |u(y) - u^\star(x)|^\frac{n}{n-1}  \dd y = 0,
        \]
        where $u^\star$ is the precise representative of $u$.
         \end{corollary}
        By standard geometric measure theory results, this yields a capacitary estimate for the approximate discontinuity set:
    
    \begin{corollary}[Dimension of the discontinuity set]\label{cor:cap}
        Let $\Acal$ be a first-order order complex-elliptic operator and let $u$ be a function in $BV^\Acal(\Omega)$. Then, we have the following Riesz capacity estimate: 
        \begin{equation*}\label{eqn:capacityest}
        \mathrm{Cap}_{n-1}(S_u) = 0.
        \end{equation*}
        Here, the $s$-dimensional Riesz capacity of a set $E \subset \Omega$ is defined as
        \begin{align*}
        \mathrm{Cap}_{s}(E) \coloneqq \sup_{\substack{\mu \in \Mcal^1(\R^n)\\  \supp (\mu) \subset E}}  \Bigg\{\bigg(\int \int \frac{\dd \mu(y)}{|y - x|^{n-s}}\dd \mu(x)\bigg)^{-1}\Bigg\}.
        \end{align*}
                These size estimates are critical in the sense that 
        \[
        		\mathrm{Cap}_s(E) < \infty \quad \Longrightarrow \quad \Hcal^{s + \eps}(E) = 0 \quad \forall \eps > 0\,.
        \]
    \end{corollary}
    Motivated by the cases $\Acal=D$ and $\Acal=E$, we conjecture there is $|\Acal u|$-essentially only one type of discontinuity, namely jump-type discontinuities:
    \begin{conjecture}
    	Let $\Acal$ be a first-order order complex-elliptic operator and let $u$ be a function in $BV^\Acal(\Omega)$. Then
	\[
		|\Acal u|(S_u \setminus J_u) = 0\,.
	\]
    \end{conjecture}
The proof of this conjecture can be restated in terms of $\Acal^c u$ and its relation with the approximate discontinuity set: One would like to show that $u$ is approximately continuous $|\Acal^c u|$-almost everywhere. It is worth mentioning that aside from the $BV$ theory, where it is known the Cantor part of $Du$ sits on a superposition of $(n-1)$-rectifiable sets, very little is known about the structure of the Cantor part $E^c u$ for functions of bounded deformation. Even less is currently known for $\Acal^c u$ when $u$ is a function of bounded $\Acal$-variation; understanding this part of the measure $\Acal u$ is, in general, a deep and difficult problem.

\subsection{Statements for higher-order operators}\label{sec:higher_order} This section extends the statements from the previous section to encompass higher-order operators. We address the identification of the density of $\Acal^a u$ with $A^k(\nabla^k u)$ through the framework of higher-order $L^p$-differentiability for elliptic operators of arbitrary order, a topic not yet fully explored in the literature. Additionally, regarding the study of jump-type discontinuities, we present a rigorous argument demonstrating that for complex-elliptic operators, only the jumps of the immediate lower-order derivatives need to be considered.\\

\noindent \textbf{Notation.} Throughout this section, we fix $k \ge 2$ as a positive integer and consider a $k$th-order elliptic operator $\Acal$ on $\R^n$, acting from a space $V$ to a space $W$.

       Let $r \ge 0$ be an integer. For a vector $e \in \R^n$, we write $e^{\otimes ^r}$ to denote the $r$-fold tensor that results by taking the tensorial product of $e$ with itself $r$ times (with the convention $e^{\otimes^0} = 1$).  Notice that
\begin{align*}
	E_r(\R^n) & = \spn \set{e^{\otimes^r}}{e\in\R^n}. 
\end{align*}
 We shall consider a linear contraction $\dpr{\frarg\,,\frarg}_r :  V \otimes E_r(\R^n) \times E_r(\R^n) \to V$ defined (with the Frobenius inner product $(\frarg \, : \, \frarg)$ on $\otimes^r \R^n$) as
\[
	\dpr{v \otimes M,e}_r = v (M : e^{\otimes^r}) \qquad \text{for all $v \in V$, $M \in E_r(\R^n)$ and  $e\in\R^n$.}
\]
Notice that under our convention for $r=0$, we have $\dpr{v,\frarg}_0 = v$. Since it will be instrumental for our computations, we will often use the alternative \emph{jet-form} expression  $A^k(D^k u) = \Acal u$, where $A^k : V \otimes E_k(\R^n) \to W$ is the (unique) linear map satisfying 
	\begin{equation}\label{eq:1111}
		A^k[v \otimes \xi^{\otimes^k} ] = \Abb^k(\xi)[v] \qquad \text{for all $v \in V$\,.}
	\end{equation}
 The existence of $A^k$ is a direct consequence of the universal property of the tensor product and the $k$-linearity of the principal symbol on the frequency variable. 
    
\subsection*{Order-reduction relations} Henceforth  we write $\Bcal$ to denote the first-order elliptic operator associated to $\mathcal A$, as defined in Theorem~\ref{1.1}.  

By classical elliptic regularity theory, there exists a natural continuous embedding $BV^\Acal_\loc \embed W^{k-1,p}_\loc$ for all $1 \le p < n/(n-1)$. If $u \in BV^\Acal_\loc$, we may then identify $D^{k-1}u$ by integration with a locally $p$-integrable map. 
Moreover, 
in this case
\[
        u \in BV^\Acal_\loc(\Omega) \qquad\Longleftrightarrow\qquad D^{k-1} u \in BV^\Bcal_\loc(\Omega)\,.
\]

    \subsection{$L^p$-differentiability} Following Ziemer~\cite{Ziemer}*{Section 3} 
    we say that a map $u$ defined on $\Omega \subset \R^n$ is $L^p$-differentiable of order $k$ at a point $x \in \Omega$ if it can be decomposed as
    \[
        u(x + h) = P_x^k(h) + R_x^k(h)\qquad \text{for all sufficiently small $h \in \R^n$\,,}
    \]
    where $P_x^k$ is a polynomial of degree at most $k$ and the reminder $R_x^k$ satisfies
    \[
        \left( \frac{1}{r^n}\int_{B_r(0)} |R_x^k(h)|^p \, dh \right)^\frac1p = \SmallO(r^k)\,.
    \]
    In this case, we define the $k$th order $L^p$-derivative $\nabla^k u(x)$ of $u$ at $x$ as the classical $k$th-order derivative of the polynomial $P_x^k$ evaluated at zero, i.e., $D^k P^k_x(0)$. 
\begin{remark}
    Every $u \in W^{k,p}_\loc$ is almost everywhere $L^p$-differentiable of order $k$. Moreover, the distributional gradient $D^k u$ can be identified by integration with the measure $\nabla^k u \, \Leb^n$ (see~\cite{Ziemer}*{Theorems~3.4.1 and~3.4.2}).  
\end{remark}
    Let $1 \le p < \infty$ and set  
    \[
        \sigma(k) \coloneqq \begin{cases}
            {n}/({n-k}) & \text{if $k < n$} \\
            +\infty & \text{if $k \ge n$}
        \end{cases}\,.
    \]
   The next lemma generalizes~\cite{GR17}*{Theorem~1.1 and Lemma~3.1} to operators of arbitrary order:
    
    \begin{lemma}[$L^p$-differentiability]\label{lem:diff} Let $\Acal$ be a $k$th-order elliptic operator. If $u$ is a map in $BV^\Acal(\Omega)$, then\vskip0.5em
        \begin{enumerate}[label=(\roman*),itemsep=0.5em]
        \item The map $u$ is $L^p$-differentiable of order $k$ for every $p$ with $1 \le p < \sigma(k)$ and $\Leb^n$-almost every $x \in \Omega$. 
        \item Denoting by $A^k$ the linear map associated with the jet expression $A^k(D^k u) = \Acal u$, we have
            \[
                    \frac{d\Acal u}{d\Leb^n}(x)  = A^k(\nabla^k u(x)) \qquad \text{for $\Leb^n$-almost every $x \in \Omega$\,.}
            \]
        \item If moreover, $k < n$ and $\Acal$ is $\C$-elliptic, then $u$ is $L^{n/(n-k)}$-differentiable of order $k$, for $\Leb^n$-almost every point in $\Omega$.  
        \end{enumerate}
    \end{lemma}

    \subsection{Jumps of higher-order gradients} This section aims to characterize the jump part of $\Acal u$. However, unlike the first-order case, we must first clarify what we mean by the jump part of $\Acal u$ for a higher-order operator. The fact that every $u \in BV^\Acal_\loc$
    belongs to $W^{k-1,1}_\loc$ suggests that only the jump points of $\nabla^{k-1} u$ are relevant for understanding the measure $\Acal u$. The next proposition shows that this is precisely the case, provided that $\Acal$ is complex-elliptic.

    \begin{proposition}\label{prop:jump} Let $\Acal$ be a $\C$-elliptic operator of order $k \ge 2$. If $u \in BV^\Acal(\Omega)$, then
        \[
        |\Acal u|(J_{\nabla^r u}) = 0 \quad \text{for all integers $1 \le r < k-1$}\,.
        \]
    \end{proposition}
    The previous proposition establishes that for a $\C$-elliptic operator of order $k$, the $|\Acal u|$-measure of the set of jump points of any lower-order derivative of $u$ is zero, when $u$ belongs to the space $BV^\Acal(\Omega)$. A natural question arises:
    \begin{open}
        Does the assertion of Proposition \ref{prop:jump} hold for general elliptic operators (not necessarily complex-elliptic)? Compare this with Theorem~\ref{lem:jump} and Open problem~\ref{op:Nin}.
    \end{open}
    
    The previous proposition motivates the following definition:
    \begin{definition}
        Let $\Acal$ be a $k$th order $\C$-elliptic operator and let $u \in BV^\Acal(\Omega)$. We define the \emph{jump part} of the measure $\Acal u$ as 
        \[
            \Acal^j u \coloneqq \Acal \mres J_{\nabla^{k-1} u}\,.
        \]
        Notice that with this definition it holds $\Acal^j u = \Bcal^j U$. 
    \end{definition}

   We now turn to the identification of the jump part as per the previous definition. The compatibility conditions for Sobolev maps across a hyperplane say that $(A,B,\nu)$ is an admissible jump triple of order $r$ if and only if the tensors $A, B \in E_{r}(\R^n)$ are rank-one connected in the direction of $\nu$, as $E_{r-1}(\R^n)\otimes \R^n$ tensors. Due to the symmetries of the elements in $E_{r}(\R^n)$, this is equivalent to requiring that
    \begin{equation}\label{eq:rk}
    A - B  = \dpr{A-B, \nu^{\otimes^{r}}}_{k-1} \otimes \nu^{\otimes^{r}} \quad \text{for some $v \in V$.}
    \end{equation}
    This analysis has the following direct implication: 
    \begin{lemma}\label{lem:jump2}
         Let $\Acal$ be a  $\C$-elliptic operator of order $k \ge 2$. If $u \in BV^\Acal(\Omega)$, then   
         \[
                \Acal^j u = \Abb^k(\nu) \dpr{U^+ - U^-,\nu^{\otimes^{k-1}}}_{r-1} \,\Hcal^{n-1} \mres J_{U}\,, \qquad U = \nabla^{k-1} u\,,
         \]
         where $\nu$ is an orientation of $J_U$. 
    \end{lemma}

    \subsection{Related statements} It is easy to verify by means of Theorem~\ref{1.1}  that the statements contained in Sections~\ref{ss:traces}-\ref{ss:app-cty} also hold for complex-elliptic operators of arbitrary order, with the following minor modifications:\vskip0.5em
    \begin{enumerate}[label={(\roman*)},itemsep=0.3em]
        \item $u$ is replaced by $U = \nabla^{k-1} u$, 
        \item $J_u$ is replaced by $J_U$, and
        \item $u^+ - u^-$ is replaced by $\dpr{U^+ - U^-,\nu^{\otimes^{k-1}}}_{k-1}$.
    \end{enumerate}

    \subsection{Other properties of complex-elliptic operators}\label{ss:other} We conclude the exposition of our results with a few properties of complex-elliptic operators that are of interest in their own right. We defer the proofs to Section \ref{s:other-prop}. 
    
    \subsubsection{Removable singularities}The following result  tells us that complex-ellipticity implies that certain lower dimensional singularities are removable:

\begin{proposition}[Removable singularities]\label{prop:equiv}
        Let $n \ge 2$ and let $\Acal$ as in \eqref{eq:B} be a complex-elliptic operator.  If $K \subset \Omega$ is a closed set with $\Lcal^n(K) = 0$ and   $(\Omega \setminus K)$ is connected, then every  solution $u \in L^1_\loc(\Omega)$  to
        \[
        \Acal u = 0 \quad \text{in the sense of distributions on  $\Omega \setminus K$},
        \]
        is also a classical solution of
        \[
        \Acal u = 0 \quad \text{on $\Omega$.}
        \]
    \end{proposition}
    \begin{remark} A direct consequence of Proposition \ref{prop:equiv} is a new and elementary proof of the following properties, which have already been shown in~\cite{VS2}:
        \begin{enumerate}[label={\arabic*.}]
            \item If $\pi\le \R^n$ is a two-dimensional subspace and $\lambda \in V$, then the equation
            \[
            \Acal u = \lambda \, \Hcal^{n-2} \mres \pi  
            \]
            has no solution in the space of tempered distributions $\Scal'(\R^n;V)$. 
            
            \item $\Acal$ is canceling on $2$-planes, that is,
            \begin{equation*}
            \bigcap_{\xi \in \pi} \im \Abb^k(\xi) = \{0\} \qquad \text{for all $\pi \in \Gr(2,n)$.}
            \end{equation*}
        \end{enumerate}
    \end{remark}

    \subsubsection{Minimal number of equations for a complex-elliptic operator}Another interesting question concerning the algebraic structure of operators satisfying a mixing condition has been discussed in~\cite{van-schaftingen2013limiting-sobole}*{Eqn.~6.1} regarding an open question by \textsc{Bourgain} and \textsc{Brezis} (see Open problem 3 in~\cite{BB2}). There, it is shown that if $\Acal$ is a first-order elliptic operator in $\R^n$-variables, from $V$ to $W$, and $\Acal$ is canceling: 
    \[
    \bigcap_{\xi \in \R^n} \im \Abb(\xi) = \{0_W\},
    \]
    then 
    \[
        \max\{n,\dim(V)+1\} \le \dim(W) \le \dim(V) + n - 1\,.
    \]
    We demonstrate that this suggested upper bound is a \emph{sharp} lower bound for the number of equations of a first-order complex-elliptic system on $n$ variables: 
    \begin{proposition}\label{prop:alg}
        Let $\Acal$ be a first-order elliptic operator in $\R^n$-variables, from a space $V$ to a space $W$.
        Then
        \[
        \dim(W) \ge \dim(V) + n -1.
        \]
        Moreover, this bound is sharp in the following sense: there exists a complex-elliptic operator on $\R^n$, from $\R^M$ to $\R^{M + n - 1}$.
    \end{proposition}

\subsection{Structure of paper}
    Section \ref{sct:prelim} below is dedicated to the preliminaries required for the proofs of the main results. Then, Section \ref{s:str-proofs} contains the proofs of the main structural results stated in Sections \ref{ss:jump}, \ref{ss:traces} and \ref{ss:jump-C}. We prove the approximate continuity results stated in Section \ref{ss:app-cty} in Section \ref{sct:Cantor_dim_est}. In Section \ref{s:other-prop}, we provide the proofs of the additional properties stated in Section \ref{ss:other}. We conclude with some examples in Section \ref{sct:ex}.

    \subsection*{Acknowledgments} We would like to thank Eduardo Simental and Giacomo del Nin to whom we are especially indebted for their numerous comments. We are also very grateful to Robert Kohn for sending us a copy of his thesis, which has been a crucial source of information for us.

    \section{Preliminaries}\label{sct:prelim}
       We begin this section by a recollection of some basic definitions and notation.
    
    \subsection{General notation}
    Recall that our spaces $V$ and $W$ are finite dimensional. We will let $M$,$N$ denote the dimensions of $V$ and $W$ respectively. We fix a 
    basis $\{e_1,...,e_M\}$ of $V$ and take $X$ to be a finite dimensional Euclidean space.

    \subsection{Basic geometric measure theory and functional analysis} 
        By the Riesz Representation Theorem, the space $\Mcal(\Omega;X)$ is identified with the dual space of $\Crm_0(\Omega;X)$, and the duality pairing is realized
        via integration as follows:
    \[
    \langle \mu, \phi \rangle := \int_\Omega \phi \dd\mu, \qquad \mu \in \Mcal(\Omega;X), \ \phi \in \Crm_0(\Omega;X).
    \]
    We hence naturally endow the space $\Mcal(\Omega;X)$ with the weak-$*$ topology. 
   
    \subsubsection{Tangent measures} For a given Radon measure $\mu \in \Mcal(\Omega;X)$, any given point $x \in \Omega$ and any $r > 0$, we define the
        \emph{re-scaled} push-forward measures $\Trm_{x,r}[\mu]$ to be the measures given by
    \[
    \Trm_{x,r}[\mu](B) \coloneqq \mu(x + rB), \quad B \in \Bfrak(\R^n),
    \]
    where $\Bfrak(\R^n)$ is the Borel $\sigma$-algebra of subsets of $\R^n$. 
    Notice that since $\Omega$ is open, this is well-defined on any set $B \in \Bfrak(\R^n)$ for $r$ sufficiently small. Following the definition of \textsc{Preiss}, we recall that a \emph{tangent measure} of $\mu$ at $x$ is a \emph{non-zero} measure $\nu \in \Mcal_\loc(\R^n;X)$ for which there exist  sequences of positive numbers $c_j$ and positive radii $r_j \searrow 0$ such that
    \[
    c_j \Trm_{x,r_j} \mu \toweakstar \nu \quad \text{in} \ \Mcal(\R^n;X).
    \]
    The space of tangent measures of $\mu$ at $x$ is denoted by $\Tan(\mu,x)$.  A fundamental result of \textsc{Preiss} \cite{preiss1987geometry-of-mea}*{Theorem~2.5} is that that $\Tan(\mu,x) \neq \varnothing$ for $|\mu|$-almost every $x \in \Omega$.
    To simplify terminology, we will use the standard notation
    \[
    \theta^{*s}(\mu, x) \coloneqq \limsup_{r \todown 0} \frac{\mu(B_r(x))}{r^s}, \qquad \theta_*^s(\mu, x) \coloneqq \liminf_{r \todown 0} \frac{\mu(B_r(x))}{r^s}
    \]
    for the respective $s$-dimensional upper and lower densities of a non-negative Radon measure $\mu$ at $x$. For any additional basic measure-theoretic notions, we refer the reader to~\cite{mattila1995geometry-of-set}, for example.

    \subsubsection{Riesz $s$-capacity}\label{sec:cap} In section \ref{sct:Cantor_dim_est}, we will also be using some basic potential theory. We formally define the \emph{$s$-Riesz potential}, $s > 0$, of a positive real-valued measure $\mu \in \Mcal^+(\Omega)$ by
    \[
    I_s(\mu)(x) := \int_\Omega \frac{1}{|x-y|^{n-s}}\dd\mu(y), \qquad x \in \Omega.
    \]
    Of course, there is no reason why this should be finite at any given point $x$. 
    Moreover, recall that for $s > 0$, the \emph{Riesz $s$-capacity} of a set $E \subset \R^n$ is defined by
       \begin{align*}\label{def:cap}
    \mathrm{Cap}_{s}(E) & \coloneqq \sup \set{\bigg(\int I_{n-s}(\mu)(x)\dd \mu(x)\bigg)^{-1}}{\ \mu \in \Mcal^1(\R^n), \ \supp \mu \Subset E.}
    \end{align*}
    If $E \in \Bfrak(\R^n)$ with $\mathrm{Cap}_s(E) = 0$, then $\Hcal^t(E) = 0$ for all $s < t < \infty$; For this and other facts see~\cite{mattila1995geometry-of-set}*{Chapter~8}). However, by \cite{Federer1972_slices_and_potentials}*{Section~2(5)}, we have
    \begin{equation}\label{eq:capacity_badset}
    \mathrm{Cap}_{n-s}\Big(\{x \in \Omega : I_s(\mu)(x) = \infty\} \Big) = 0.
    \end{equation}

    The following lemma provides an elementary estimate for the localized $m$-Riesz potential of a positive Radon measure $\mu$ on annuli, weighted with an appropriate scaling. The proof is identical to that in~\cite{kohn1979new-estimates-f}*{Lem.~5.11}. 
    \begin{lemma}\label{lem:Kohnpotential}
        Let $\mu \in \Mcal^+(\Omega)$ and suppose that $B_1 \subset \Omega$. Then, for any $m \in \N$ it holds that
        \[
        \limsup_{r\todown 0} r^m\int_{B_1\setminus B_r} \frac{1}{|y|^{n-1+m}} \dd \mu(y) \lesssim_{n,m} \theta^{*(n-1)}(\mu,0).
        \]
    \end{lemma}

    \subsection{Properties of complex-elliptic operators}\label{sec:properties} In this section we recall a few facts from the theory of complex-elliptic operators. 
        The notion of complex-ellipticity originated with the work of \textsc{Aronszajn}~\cite{aron} and \textsc{Smith}~\cites{smith1,smith2}, who used it to derive $L^p$-coercive inequalities for elliptic boundary-value problems. Recently, \textsc{Breit, Diening \& Gmeineder}~\cite{breit2017traces} have shown that, for first-order operators, complex-ellipticity is a necessary and sufficient condition for extension and trace properties of the operator. An operator $\Acal$ being complex-elliptic allows one to exploit Hilbert's Nullstellensatz to show that (on an open connected set $\Omega$) the null space   is finite dimensional subspace of $V$-valued polynomials (see~\cite{smith1}*{Cor.~8.13 and~Rmk.~4}, see also~\cite
    {breit2017traces}*{Thm. 2.6}). More precisely, there exists a positive integer $\ell = \ell(\Acal)$ such that
    \[
    \Acal u = 0 \; \text{in  $\Dcal'(\Omega;W)$ \quad $\Longrightarrow \quad u \in \set{p \in V \otimes \Rbb[\mathbf x]}{\deg(p) < \ell}$.}
    \]
    In fact, due to the complex-ellipticity of $\Acal$, any \emph{smooth} function $u \in BV^\Acal(\Omega)$ possesses a strong Taylor expansion, as observed by Ka\l amajska~\cite{kal}:
    \begin{proposition}\label{prop:Kal_int_rep}
        Let $\Acal$ be a complex-elliptic operator. Then there exists $\ell \in \N$ such that for any ball $B \subset \Omega$ and every $u \in BV^\Acal(\Omega)\cap\Crm^\infty(\Omega;V)$, we have the integral representation
        \[
        		u(y) = (\Pcal_B u)(y) + 
    \int_B K(y,z) \Acal \, u(z) \dd z, \qquad y \in B,
        \]
        where 
        	\[
        \Pcal_B u(y) := \int_B \sum_{|\beta| \leq \ell -1} \partial^\beta_z\bigg(\frac{(z-y)^\beta}{\beta !}w_B(z)\bigg) \, u(z)\dd z 
        \]
        is the averaged Taylor polynomial of $u$ with respect to a weight $w_B \in \Crm_c^\infty(B)$ with $\int_B w_B = 1$, 
        and $K \in \Crm^\infty(\R^n \times \R^n \setminus \{y=z\} ; V \otimes W^*)$ is a kernel satisfying the growth condition
        \begin{equation}\label{mierda}
        |\partial^\alpha_y\partial^\beta_z K(y,z)| \lesssim |y-z|^{-(n-1)-|\alpha|-|\beta|}, \qquad y, z \in B,
        \end{equation}
        for all multi-indices $\alpha, \beta \in \Nbb^n$.
    \end{proposition}
    
    We will use the following definition for the lower order polynomials in the above representation.
In particular, the decay properties of $K$ give us the pointwise estimate
    \begin{equation}\label{eq:polyn_approx}
    |u - \Pcal_B u|(y) = \left|\int_B K(y,z) \, \Acal u(z) \dd z \right| \lesssim \int_B \frac{|\Acal u|(z)}{|y-z|^{n-1}}\dd z.
    \end{equation}
    We will henceforth denote $\Pcal_{B_r(x)} u$ and $\Pcal_{B_r} u$ by $\Pcal_{x,r} u$ and $\Pcal_r u$ respectively, for ease of notation. 
    \subsubsection{Poincar\'e inequalities} We recall from~\cite{GR17} that $BV^\Acal(B) \embed L^{1^*}(B;V)$ and 
    \begin{equation}\label{eq:poincare1}
    \|u - \Pcal_B u\|_{L^{1^*}(B)} \le c(n,B,\Acal)|\Acal u|(B)\qquad \forall u \in BV^\Acal(B).
    \end{equation}

    We will also frequently be restricting our considerations to smooth functions $u \in BV^\Acal(\Omega) \cap \Crm^\infty(\Omega;V)$. This is because one may approximate an arbitary map in $BV^\Acal$ in the appropriate sense by such functions. Indeed, we have the following result concerning \emph{strict density} of smooth functions in the space $BV^\Acal$:
    
    \begin{proposition}\label{prop:smoothapprox}
        Let $\Omega\subset \R^n$ be an open bounded  set and let  $u \in BV^\Acal(\Omega)$. Then, there exists a sequence $(u_j) \subset BV^\Acal(\Omega)\cap\Crm^\infty(\Omega;V)$ such that
        \[
        u_j \tolong u \quad \text{in} \ L^1(\Omega), \qquad |\Acal u_j|(\Omega) \tolong |\Acal u|(\Omega).
        \]
    \end{proposition}
    The proof of this is almost identical to that for the classical $BV$ case, so is omitted here. For the details, see \cite{breit2017traces}*{Thm.~2.8} (the more general approximation result for constant rank operators is also given in~\cite{arroyorabasa2019characterization}*{Thm.~1.6}).

    \section{Proofs of the structural properties}\label{s:str-proofs}
    
    This section is devoted to the proof of all the results previously discussed with exception of the capacitary estimate of $S_u$.

    \subsection{Proof of Theorem~\ref{lem:jump}} First, we show that $J_u \subset \Theta_u$. In fact we shall prove that the lower $(n-1)$-dimensional density is non-zero. We follow the classical reasoning used for $BV$ and $\BD$-spaces; see~\cite{ambrosio1997fine-properties}.
    Let $x \in J_u$, and fix a scale $r>0$ sufficiently small so that $B_r(x) \subset \Omega$. By the properties of $J_u$, for the re-scaled functions
    \[
    u_r \coloneqq u(x + r\, \frarg) : B_1 \to V,
    \]
    there exist $a, b \in V$ such that 
    \[
    |a-b| > 0 \quad \text{and} \quad u_r \tolong \mathbf 1_{(a,b,\nu(x))} \;\; \text{in $L^1(B_1;V)$},
    \]
    where as usual, $\nu(x)$ is a unit normal to $J_u$ at $x$, and
    \[
        \mathbf 1_{(a,b,\nu)}(x) \coloneqq  \begin{cases}
            a & \text{if $x \cdot \nu > 0$}\\
            b & \text{if $x \cdot \nu \le  0$}\\
        \end{cases}.
    \]
    The lower semicontinuity of the map $v \mapsto |\Acal v|$ on $B_1$ with respect to $L^1$-convergence and a change of variables yield
    \begin{align*}
    \liminf_{r \todown 0} \frac{|\Acal u|(B_r(x))}{r^{n-1}} & = \liminf_{r \todown 0} |\Acal u_r|(B_1) \\
    & \ge |\Acal \mathbf 1_{(a,b,\nu(x))}|(B_1) \\
    & \simeq_n |\Abb(\nu(x)) (a - b)| \\
    & \ge c_\Acal |a-b| >0,
    \end{align*} 
    where in the last two-inequalities we have used that $\mathbf 1_{(a,b,\nu(x))}$ belongs to $BV(B_1;V)$ and that the ellipticity constant $c_\Acal$ is positive. 
    This shows that $x \in \Theta_u$. 
    
    We now turn to the rectifiability question. By classical measure theoretic arguments (see, for example, \cite{mattila1995geometry-of-set}*{Sec.~6}), it follows that the set 
    \[
    \set{x \in J_u}{\theta^{*(n-1)}(|\Acal u|,x) = \infty} \subset J_u
    \] 
    is $\Hcal^{n-1}$-negligible. Let us define $G_u \coloneqq \set{x \in J_u}{\theta^{*(n-1)}(|\Acal u|,x)< \infty}$. 
        Notice that since we also have $J_u \subset \Theta_u$, the lower and upper dimensional densities are non-degenerate for all points in $G_u$, namely, 
    \[
    0 < \theta_*^{n-1}(|\Acal u|,x) \le  \theta^{* (n-1)}(|\Acal u|,x) < \infty, \qquad x \in G_u.
    \]
    In particular, the measures $\Hcal^{n-1}$ and $\Acal u$ are \emph{equivalent} on $G_u$, i.e.,
    \begin{equation}\label{eq:G}
    \Acal u \mres G_u \ll \Hcal^{n-1} \mres G_u \ll \Acal u \mres G_u.
    \end{equation}
    In this regime, one may replace the normalizing constants $c_j \todown 0$ of every blow-up sequence  
    \[
    c_j T_{x,r_j}[\Acal u] \toweakstar \tau, \quad r_j \todown 0;
    \] 
    by $c_j = c r_j^{-{(n-1)}}$ (up to subsequence) for some positive number $c > 0$. 
    Now, we already know that
    \[
    c \ \Acal u_{r_j}=  c_j T_{x,r_j}[\Acal u],
    \]
    and that
    \[
    \Acal u_{r_j} \toweakstar \Acal \mathbf 1_{(a,b,\nu(x))} = \Abb(\nu(x))(a-b)\Hcal^{n-1} \mres \nu(x)^\perp.
    \]
    Thus, we must have
    \[
    \tau = c \ \Abb(\nu(x))(a-b)\Hcal^{n-1} \mres \nu(x)^\perp,
    \]
        which is a uniform measure over the hyperplane $\nu(x)^\perp \in \Gr(n-1,n)$. Since $\tau$ was arbitrary tangent measure of $\Acal u$, this calculation shows (cf.~\eqref{eq:G}) that 
    \begin{equation}\label{tr}
    \Tan(\Acal u,x) = \Gcal_{n-1,n}(\nu(x)^\perp) \quad \text{for $|\Acal u|$-almost every $x \in G_u$}.
    \end{equation}
    Here, $\Gcal_{m,n}(\pi)$ is the set of $m$-flat measures supported on $\pi \in \Gr(m,n)$. 
    We can now apply the rectifiability criterion contained~\cite{mattila1995geometry-of-set}*{Theorem~16.7}, which states that
    \[
    \Acal u \mres G_u \:\: \text{is countably $\Hcal^{n-1}$-rectifiable,}
    \]
    and also addresses the desired countable $\Hcal^{n-1}$-rectifiability of $G_u$ (and $J_u$) up to an $\Hcal^{n-1}$-null set. Notice that, up to a change of sign, the characterization in~\eqref{tr} also implies that the Borel map $\nu$ (from the Borel jump triplet $(u^+,u^-,\nu)$) is an orientation of $J_u$.   Moreover, the characterization of the tangent measure $\tau$ discussed above, and the classical measure theoretic fact that $\Tan(\Acal u,x) = \frac{\dd \Acal u}{\dd |\Acal u|}(x) \Tan(|\Acal u|,x)$ for $|\Acal u|$-almost every $x \in \Omega$, implies that
    \[
    \frac{\dd \Acal u}{\dd |\Acal u|}(x) = \Abb(\nu(x))\big(u^+(x) - u^-(x)\big)
    \]
    This verifies the sought representation of $\Acal^j u \mres G_u$. \qed

    \subsection{Proof of Theorem~\ref{thm:twoside}} We prove the statement of the theorem when $M$ is the graph of a Lipschitz map $f : \R^{n-1} \to \R$ on $\Omega$. The statement for general countably $\Hcal^{n-1}$-rectifiable sets then follows by a standard geometric covering argument. 
    Let us write $\Omega^+ = \set{z \in \Omega}{z > f(x)}$ and  $\Omega^- = \set{z \in \Omega}{z < f(x)}$ to denote the (open and locally Lipschitz) sides of $M$ on $\Omega$. In this case, \cite{breit2017traces}*{Cor.~4.21} applied to the map $u = u\mathbf 1_{\Omega^+} + u\mathbf 1_{\Omega^-}$ gives
    \[
    \Acal u = \Acal u \mres \Omega^+ + \Acal u \mres \Omega^- + \Abb(\nu_\Gamma)(\tr^+(u) - \tr^-(u)) \, \Hcal^{n-1} \mres M,
    \]
        where $\tr^\pm: BV^\Acal(\Omega^\pm) \to L^1(\Gamma;V)$ is the \emph{exterior} linear trace operator corresponding to $\partial \Omega^\pm$, and where $\nu_M(x)$ denotes the outer unit normal of $\Omega^+$ at $x$. Our candidate for the one-sided values of $u$ on $M$ will naturally be $u^\pm_M = \tr^\pm(u)$, which (by the boundedness of the one-sided exterior traces) exist for $\Hcal^{n-1}$-almost every $x \in M$. 
    Once this is verified, we will obtain the desired expression
    \begin{equation}\label{eq:middle}
    \Acal u \mres M = \Abb(\nu_M)\big(u_M^+ - u_M^-\big) \, \Hcal^{n-1} \mres M.
    \end{equation}
    We are left to check that $u^\pm_M$ are in fact the two-sided approximate limits of $u$. Without loss of generality, we shall show only that
    \[
    \limsup_{r\todown 0} \aveint{B_r(x) \cap \Omega^+}{} |u - u_M^+(x)|^{1^*} \dd y = 0 \quad \text{for $\Hcal^{n-1}$-almost every $x \in M$},
    \]
    since the proof for $u^-_M$ is entirely analogous.
    
    \textit{Step~1. Removal of discontinuities on the surface.} Since $\tr^+(u) \in L^1(M;V)$, we may use the classical $BV$-extension to find $v \in BV(\Omega^-,V)$ satisfying 
    \begin{equation*}\label{eq:taylor0}
    \tr^-(v) = u_M^+ \quad \text{on $L^1(M;V)$.}
    \end{equation*}
   The classical $BV$ trace operator $\tr^-: BV(\Omega^-;V) \to L^1(M;V)$ satisfies the point-wise average representation (see for instance~\cite{AFP2000}*{Thm.~3.87})
    \begin{equation}\label{eq:taylor2}
    \limsup_{r \todown 0} \aveint{B_r(x) \cap \Omega^-}{} |v - u^+_M|^{1^*} \dd y = 0 \quad \text{for $\Hcal^{n-1}$-almost every $x \in M$.}
    \end{equation}
    Define $\tilde u \coloneqq \mathbf 1_{\Omega^+} u + \mathbf 1_{\Omega^-} v$. It follows from~\cite{breit2017traces}*{Corollary~4.21} that $\tilde u \in BV^\Acal(\Omega)$. Furthermore, due to the compatibility conditions on $M$ we get 
    \[\Acal \tilde u \mres M \equiv 0, \quad \text{or equivalently,} \quad \Acal \tilde u = \Acal u \mres \Omega^+ + \Acal v \mres \Omega^-. 
    \]
    \textit{Step~2. Polynomial approximation.} Let us recall the following well-known property of mutually singular measures: the Radon--Nykod\'ym Differentiation Theorem implies that
    \[
    \frac{\dd \sigma}{\dd \nu}   = 0 \quad \text{for $\nu$-almost every $x \in M$,}
    \]
    whenever $\sigma \perp \nu$. 
    This property applied to $\sigma = |\Acal \tilde u|$ and $\nu =\Hcal^{n-1} \mres M$ in turn gives 
    \begin{equation*}\label{eq:density}
    \limsup_{r \todown 0} \frac{|\Acal \tilde u| (B_r(x))}{r^{n-1}} = 0 \quad \text{for $\Hcal^{n-1}$-almost every $x \in M \cap \Theta_{\tilde u}$.}
    \end{equation*}
    Note that we are considering the set of positive density of $\tilde u$ in the above statement. In particular, the \emph{scale-dependent continuity} (cf. Theorem~\ref{prop:capacity}) of functions of bounded $\Acal$-variation at points where the upper $(n-1)$-dimensional density vanishes says that for $\Hcal^{n-1}$-almost every $x \in M$ there exist $V$-valued polynomials $\Pcal_{x,r} \tilde u$ in a finite dimensional subspace $\mathscr F \le V \otimes \R[x_1,\dots,x_n]$ with the property that
    \begin{equation}\label{eq:taylor1}
    \limsup_{r \todown 0} \aveint{B_r(x)}{} |\tilde u - (\Pcal_{x,r} \tilde u)_0|^{1^*} \dd y = 0.
    \end{equation}
       Here, we used~\eqref{eq:quasihot} in Theorem~\ref{prop:capacity} to dispense with the higher order terms of $\Pcal_{x,r}\tilde{u}$ in the limit as $r \to 0$. From~\eqref{eq:taylor2} and~\eqref{eq:taylor1} it follows that the limit of the zero-order term must be $u_M^+(x)$, that is
    \begin{align*}\label{eq:tracehot}
 \limsup_{r \todown 0}  & |(\Pcal_{x,r}\tilde{u})_0 - u_M^+(x)| = \limsup_{r \todown 0}   \aveint{B_r^-(x) \cap \Omega^-}{}  |(\Pcal_{x,r} \tilde u)_0- u_M^+(x)|\dd y  \\
   & \lesssim_n  \limsup_{r \todown 0}   \left[ \aveint{B_r(x) \cap \Omega}{} |(\Pcal_{x,r} \tilde u)_0- \tilde u|^{1^*}\dd y +  \aveint{B_r(x) \cap \Omega^-}{} |v - u^+_M|^{1^*}\dd y \right] = 0.
      \end{align*}
   We conclude that
    \begin{align*}
    \limsup_{r \todown 0} & \aveint{B_r(x) \cap \Omega^+}{} |u - u_M^+(x)|^{1^*} \dd y = \limsup_{r \todown 0}\aveint{B_r(x) \cap \Omega^+}{} |\tilde{u} - (\Pcal_{x,r}\tilde{u})_0|^{1^*} \dd y \\
   & \lesssim_n \limsup_{r \todown 0} \aveint{B_r(x)}{} |\tilde{u} - (\Pcal_{x,r}\tilde{u})_0|^{1^*} \dd y \stackrel{\eqref{eq:taylor1}}= 0.
    \end{align*}
    The $L^1$-bound for $u^+_M - u^-_M$ follows directly from the ellipticity of $\Acal$ and~\eqref{eq:middle}:
    \[
    	C(\Abb) \int_M |u^+_M - u^-_M| \dd \Hcal^{n-1} \le \int_M |\Abb(\nu)[u^+_M - u^-_M]| \dd \Hcal^{n-1} \le |\Acal u|(\Omega).
    \]
    
    \subsection{Proof of Corollaries \ref{exterior} and \ref{interior}} The characterization of the exterior trace follows from the characterization of $\tr_\Omega$ in terms of $u^+_\Omega$ and its continuity bound
    \[
    	\|\tr_\Omega u\|_{L^1(\partial \Omega)} \le C \|u\|_{BV^\Acal(\Omega)}\,.
    \]
The proof of the boundedness of the interior trace follows verbatim from the same argument applied to $\tr_{\Omega^+}$ and $\tr_{\Omega^-}$, where $\Omega^\pm$ are as in the previous proof. \qed

    \subsection{Proof of Corollary~\ref{thm:jump2}} The first assertion was already proved in~\eqref{eq:middle} for Lipschitz hypersurfaces. The general case for a countably $\Hcal^{n-1}$-rectifiable set $\Gamma$ follows from the fact that the Lipschitz surfaces $(M_h)_{h\in \Nbb}$ ``covering'' can be taken such that 
    \[
    	\Hcal^{n-1}(M_i \cap M_j) = 0 \qquad \text{for all $i \neq j$}\,.
    \]
    
The following is a well-known fact of tangents to rectifiable sets: if $\Gamma_1, \Gamma_2$ are countably $\Hcal^m$-rectifiable, then
\[
	\Tan(\Gamma_1,x) = \Tan(\Gamma_2,x) \qquad \text{for $\Hcal^{m}$ a.e. $x \in \Gamma_1 \cap \Gamma_2$}\,.
\]
The second statement follows directly from Theorem~\ref{thm:twoside} by letting $m = n-1$ and taking $\Gamma_1 = \Gamma$ and $\Gamma_2 = J_u$. Indeed, we may assume without loss of generality that $\nu$ orientates both $\Gamma$ and $J_u$ at their intersection, and hence, by Theorem~\ref{thm:twoside} it follows that $(u^+,u^-,\nu)$ is a jump triple.

    Lastly, we show that $S_u \setminus J_u$ is $\Hcal^{n-1}$-purely unrectifiable. Let $M \subset \Omega$  be an $m$-dimensional Lipschitz hypersurface.  Since $M$ is rectifiable, the set $M \cap S_u \cap J_u^\complement$ is also rectifiable. Therefore, $u$ has one-sided limits for $\Hcal^{n-1}$-almost every $x \in M \cap S_u \cap J_u^\complement$. However, the assumption $x \notin J_u$ implies the one-sided limits coincide and hence $u$ is approximately continuous at $\Hcal^{n-1}$-almost every $x$ there. 
On the other hand, $x \in S_u$, so the previous statement can only hold on a negligible set, that is, 
    \[
    \Hcal^{n-1}(M \cap S_u \cap J_u^\complement) =0.
    \]
    This proves that $S_u \cap J_u^\complement$ is $\Hcal^{n-1}$-purely unrectifiable. \qed

    \section{Dimensional estimates for $S_u$}\label{sct:Cantor_dim_est}
    
    Here and in what follows we assume that $\Acal$ is complex-elliptic.
        The purpose of this section is to establish a dimensional estimate on $S_u \setminus \Theta_u$ of all discontinuity points with zero density. Moreover, we show that all points of zero upper $(n-1)$-density vanishes are in fact quasi-continuity points.
    
    To begin with, let us introduce the following notation. Write
    \begin{equation}\label{eq:capacityblowup}
    N_u  \coloneqq \set{x \in \Omega}{I_1 (|\Acal u|)(x) = \int |x-y|^{-(n-1)} \dd|\Acal u|(y)= \infty}.
    \end{equation}
    By \eqref{eq:capacity_badset}, we know that $\mathrm{Cap}_{n-1}(N_u) = 0$.

    \subsection{Proof of Theorem~\ref{prop:capacity}}
    We need to prove the decay of the higher-order Taylor coefficients and the final statement of approximate continuity, since the proof of the first estimate in the proposition follows from an integrated version of estimate \eqref{eq:polyn_approx} and an approximation argument as in Proposition~\ref{prop:smoothapprox}. We adopt an analogous approach to~\cite{kohn1979new-estimates-f}*{Sect.~5}. We show the following lemma, from which Theorem \ref{prop:capacity} follows immediately. 
    
    \begin{lemma}\label{L:1}
        Suppose that $\Acal$ is complex-elliptic, and let $u \in BV^\Acal(\Omega)$.
        For any multi-index $\alpha \in \Nbb^n$, let $c_{x,r}^\alpha/\alpha! \in V$ to be the coefficient of the term $(y - x)^\alpha$ from the Taylor polynomial $\Pcal_{x,r} u$. Then, 
        \[
        \limsup_{r \todown 0} \, r^{|\alpha|}|c_{x,r}^\alpha| \lesssim_{n,\alpha}  \theta^{*(n-1)}(|\Acal u|,x) \qquad \text{for all $1 \le |\alpha| \le \ell-1$},
        \]
        where $\ell \in \Nbb$ is the integer from Proposition~\ref{prop:Kal_int_rep}. Moreover, 
                \[
        \limsup_{r \todown 0} |c_{x,r}^0 - c_{x,\rho}^0| \lesssim_{n} \int_{\cl{B_r(x)} \setminus B_\rho(x)} \frac{|\Acal u|(y)}{|y - x|^{n-1}} \dd y
        \]
         for all $0 < r \le \rho < \dist(x,\partial\Omega)$.
            \end{lemma}
    
    \begin{proof}[Proof of Lemma~\ref{L:1}]
        Fix a multi-index $\alpha$ as in the statement of the lemma. Without any loss of generality, we may assume that $x = 0$.  First, we show that the assertions hold for any given $u \in BV^\Acal(\Omega) \cap \Crm^\infty(\Omega;V)$.	For any $0 < t < \dist(0,\partial\Omega)$, we have
        \[
        \partial^\alpha u(y) = \partial^\alpha\Pcal_t u(y) + \int_{B_t} \partial^\alpha_y K(y,z) \, \Acal u(z) \dd z.
        \]
        In particular, since $\Pcal_t u$ is a polynomial centered at $0$, it must hold that
        \[
        \partial^\alpha u(0) =  c_{t}^\alpha + \int_{B_t} \partial^\alpha_y K(0,z) \, \Acal u(z) \dd z, 
        \]
        where ${c_t^\alpha}/{\alpha!}$ is the coefficient of the monomial $x^\alpha$ in the Taylor polynomial $\Pcal_t u$.
        Thus, taking differences between coefficients at scales $0 < r < \rho < \dist(0,\partial\Omega)$ and multiplying by $r^{|\alpha|}$, we obtain the estimate (cf.~\eqref{mierda})
                \begin{equation}\label{hey}
        \begin{split}
        r^{|\alpha|} |c_\rho^\alpha - c_r^\alpha| & \le 
        r^{|\alpha|}\int_{B_\rho\setminus B_r} |\partial^\alpha_y K(0,z)|\, |\Acal u(z)| \dd z \\ 
        &  \lesssim_n  r^{|\alpha|} \int_{\cl{B_\rho}\setminus B_r} \frac{|\Acal u|(z)}{|z|^{n-1 + |\alpha|}}\dd z.
        \end{split}
        \end{equation}
        Hence, by Lemma~\ref{lem:Kohnpotential} (replacing $B_1$ with $B_\rho$), we deduce that
        \[
        \limsup_{r \todown 0} r^{|\alpha|}|c_\rho^\alpha - c_r^\alpha| \lesssim_{n,\alpha} \theta^{*(n-1)}(\Acal u,0) \quad \text{for all $1 \le |\alpha| \le \ell -1$}.
        \]
        Taking $\rho = 1$ and letting $r \todown 0$ gives
        \begin{equation*}
        \limsup_{r \todown 0} r^{|\alpha|}|c_{r,\alpha}| \leq  \limsup_{r \todown 0} r^{|\alpha|}|c_{r,\alpha} - c_{1,\alpha}| + \limsup_{r \todown 0} r^{|\alpha|} |c_{1,\alpha}| = \theta^{*(n-1)}(|\Acal u|,0)
        \end{equation*}
        for all multi-indexes $\alpha$ with $1 \le |\alpha| \le \ell -1$. This proves the first assertion. The second assertion follows by taking $\alpha = 0$ in~\eqref{hey}.
        
        To prove the desired results for arbitrary $u \in BV^\Acal(\Omega)$, it suffices that the key estimate~\eqref{hey} holds for any $u \in (C^\infty \cap BV^\Acal)(\Omega)$. The argument is as follows: By Proposition~\ref{prop:smoothapprox}, we know there exists a sequence $(u_j) \subset \Crm^\infty\cap BV^\Acal(\Omega)$ with $u_j \to u$ in $L^1$ and $|\Acal u_j|(\Omega) \to |\Acal u|(\Omega)$. Let $c^{u_j}_{r,\alpha}/\alpha!$ (respectively $c^u_{r,\alpha}/\alpha!$) denote the coefficient of the order $\alpha$ term of $\Pcal_r u_j$ (respectively $\Pcal_r u$). Then, since each $c^u_{r,\alpha}$ is a weighted sum of terms of the form
        \[
        \int_{B_r} \partial^\beta\big(z^{\beta- \alpha} w_{B_r}(z)\big) \, u(z) \dd z, \qquad \beta > \alpha,
        \]
        we also have that 
        \[
        c_{r,\alpha}^{u_j} \tolong c_{r,\alpha}^u \quad \text{as} \ j \to \infty \qquad \text{for each $r > 0$},
        \]
        since strong convergence implies weak convergence. Moreover, we have
        \[
        \lim_{j \to \infty} \int_{\cl{B_{\rho}}\setminus B_{r}} \frac{|\Acal u_j|(z)}{|z|^{n-1+|\alpha|}} \dd z = \int_{\cl{B_{\rho}} \setminus B_{r}} \frac{1}{|z|^{n-1+|\alpha|}} \dd |\Acal u|(z),
        \]
        due to the strict convergence. Thus~\eqref{hey} indeed extends to all functions of bounded $\Acal$-variation.
    \end{proof}

    \subsection{Proof of Corollary~\ref{thm:cap_est}} 
    We claim that
    \[
    S_u \subset N_u \cup \Theta_u,
    \]	
    where $N_u$ is defined as in~\eqref{eq:capacityblowup}. In fact we show the stronger property that if 
    \[
    \int_{B_r(x)} \frac{|\Acal u|(y)}{|y - x|^{n-1}} \dd y < \infty \quad \text{for some $r \in \dist(x,\partial\Omega)$,}
    \]
    then $x \in S_u^\complement$.
        \begin{proof} Let $c_{x,r}^\alpha$ be the coefficients of Lemma \ref{L:1}.  First, we observe that $(c_{x,r}^0)_r$ is a Cauchy sequence in $V$. Indeed, this follows directly from the second assertion in Lemma~\ref{L:1} and the absolute continuity of the function
        \[
        \eta(r) \coloneqq \int_{B_r(x)} \frac{|\Acal u|(y)}{|y - x|^{n-1}} \dd y < \infty, \qquad 0 <  r < \dist(x,\partial\Omega).
        \]
        (Notice that in particular $\theta^{*(n-1)}(|\Acal u|,x) = 0$.)
        Therefore, there exists a fixed vector $c^0_x \in V$ such that 
        \begin{equation}\label{eq:YA}
        \lim_{r\todown 0} | c^0_{x,r} - c^0_x| = 0. 
        \end{equation}
        We will now show that $u$ is $L^{1^*}$-approximately continuous at $x$. We have 
        \begin{align*}
        \bigg(\aveint{B_r(x)}{}&|u - c^0_{x}|^{1^*} \bigg)^\frac{1}{1^*}\leq \bigg(\aveint{B_r(x)}{} |u -  c^0_{x,r}|^{1^*}\bigg)^\frac{1}{1^*} + \bigg(\aveint{B_r(x)}{} | c^0_{x,r} - c_x^0|^{1^*}\bigg)^\frac{1}{1^*} \\
        &\leq \bigg(\aveint{B_r(x)}{} |u - \Pcal_{x,r} u|^{1^*}\bigg)^\frac{1}{1^*} + \bigg(\aveint{B_r(x)}{} |\Pcal_{x,r} u -  c^0_{x,r}|^{1^*}\bigg)^\frac{1}{1^*} \\& \qquad + \bigg(\aveint{B_r(x)}{} | c^0_{x,r} - c^0_x|^{1^*}\bigg)^\frac{1}{1^*} \\
        &\stackrel{\eqref{eq:poincare1}}\lesssim_{n,\Acal} \frac{|\Acal u|(B_r(x))}{r^{n-1}} + \sum_{1 \leq |\alpha| \leq \ell - 1} r^{|\alpha|}|c_{x,r}^\alpha|  + \bigg(\aveint{B_r(x)}{} | c^0_{x,r} - c_x^0|^{1^*}\bigg)^\frac{1}{1^*}.
        \end{align*}
        The first assertion of Lemma~\ref{L:1} and~\eqref{eq:YA} imply
        \[
        \lim_{r \todown 0} \bigg(\aveint{B_r}{}|u - c_x^0|^{1^*}\bigg)^\frac{1}{1^*} = 0.
        \]
        This proves that $x \in S_u^\complement$ and the claim follows. 
        
        \proofstep{Proof of Corollary~\ref{cor:cap}.} Since every point in $N_u^\complement$ is a continuity point and $\Theta_u$ is a $\sigma$-finite set with respect to the $\Hcal^{n-1}$-measure,
               proving $\mathrm{Cap}_{n-1}(S_u) = 0$ reduces to checking that
        \[
        \mathrm{Cap}_{n-1}(S_u \cap \Theta_u^\complement) = 0.
        \]
        However, this follows as a trivial corollary of the fact that $\mathrm{Cap}_{n-1}(N_u) = 0$. This finishes the proof.
    \end{proof}

    \section{Proofs of the results for higher-order operators}
     \begin{proof}[Proof of Lemma~\ref{lem:diff}] Since $\Acal$ is elliptic and the statements are local, we can assume without loss of generality that $u$ has compact support in $\Omega$ (and hence also $U$). By extending $u$ and $U$ by zero, we can assume that $u \in BV^\Acal(\R^n)$ and $U \in BV^\Bcal(\R^n)$. Since $\Bcal$ is a first-order elliptic operator, it follows from~\cite{GR17}*{Theorem~1.1 and Lemma~3.1} that $U$ is $L^p$-differentiable almost everywhere for every $1 \le p < \sigma(1)$ (or $1 \le p \le \sigma(1)$ if $\Acal$ is complex-elliptic) and
    \begin{equation}\label{eq:1st_order}
        \frac{d\Bcal U}{d\Leb^n}(x) = B(\nabla U(x)) \quad \text{$\Leb^n$-almost everywhere.}
    \end{equation}
    Since the gradient operator is complex-elliptic, an iteration of Lemma~\cite{Alberti_Bianchini_Crippa_14}*{Lemma~4.6} yields that $u$ is $k$ times $L^p$-differentiable almost everywhere for all $1 \le p < \sigma(k)$ (or $1 \le p \le \sigma(k)$ if $\Acal$ is complex-elliptic). This proves (i) and (iii). From this discussion it also stems that $\nabla U(x) = \nabla^k u(x)$ for almost every $x$; an observation that will be used in the proof the (ii). 
    
    Being the classical $k$th order derivative of a $V$-valued polynomial, it follows that $\nabla^k u(x) \in V \otimes E_k(\R^n)$ almost everywhere. Now, let $R(DU) = \Curl_{k-1}$ be the jet notation for the $\Curl_{k-1}$ operator defined in the proof of Theorem~\ref{1.1} and observe that $R$ vanishes on $V \otimes E_{k-1}(\R^n)$. Inspecting the construction of $\Bcal$, we deduce that for almost every $x$ it holds
    \[
        B(\nabla U(x)) = B(\nabla^k u(x)) = \tilde A(\nabla U(x)) \oplus R(\nabla U(x)) = \tilde A(\nabla U(x)) \oplus 0_{W^\perp}\,.
    \]
    Projecting both sides of this identity onto its $W$ coordinate, we conclude that 
    \begin{equation}\label{eq:reduction}
    P_W B(\nabla U(x)) = \tilde A(\nabla^k u(x)) = A^k(\nabla^k u(x))\quad \text{$\Leb^n$-almost everywhere\,.}
    \end{equation}
    Here, in passing to the last equality we have used that the restriction of $\tilde A$ to $V \otimes E_k(\R^n)$ coincides with $A^k$. The first statement of Theorem~\ref{1.1} and the Radon--Nykodym theorem then give
    \[
        \frac{\Acal u}{\Leb^n}(x) = P_W \frac{\Bcal U}{\Leb^n}(x) \stackrel{\eqref{eq:1st_order}-\eqref{eq:reduction}} = A^k(\nabla^k u(x)) \quad \text{$\Leb^n$-almost everywhere.}
    \]
    This proves the second assertion. 
   \end{proof}

   \begin{proof}[Proof of Proposition~\ref{prop:jump}] The classical theory of BV spaces gives $\Hcal^{n-1}(J_{\nabla^r u}) = 0$ for all $1 \le r \le k-2$. Additionally, by Remark~\ref{rem:Giacomo}, $J_{\nabla^r u}$ is countably $\Hcal^{n-1}$-rectifiable. Therefore, applying Theorem~\ref{1.1} and the first assertion of Corollary~\ref{thm:jump2}, we conclude that
    $$|\mathcal{A} u|(J_{\nabla^r u}) = |P_W\mathcal{B} U|\mres J_{\nabla^r u} \equiv 0 \quad \text{for every } 1 \leq r \leq k-2.$$
    This completes the proof.
    \end{proof}

     \begin{proof}[Proof of Lemma~\ref{lem:jump2}]
    By definition, we have $\Acal^j u = \Bcal^j U = \Bbb(\xi)[U^+ - U^-] \, \Hcal^{n-1} \mres J_U$. 
    Since $(U^+,U^-,\nu)$ is a admissible triple $\Hcal^{n-1}$-almost everywhere on $J_U$, it follows from Theorem~\ref{1.1} point (ii) and~\eqref{eq:rk} that the identity $\Bbb(\xi)[U^+ - U^-] = \Abb^k(\xi)\dpr{U^+ - U^-,\nu^{\otimes^{k-1}}}_{k-1}$ holds everywhere on $J_U$. Casting this into the expression for $\Acal^j u$ above yields the sought assertion. 
    \end{proof}
    
\section{Proofs of additional properties of complex-elliptic operators}\label{s:other-prop}
\begin{proof}[Proof of Proposition \ref{prop:equiv}]
	Using that $K$ is closed and $\R^n \setminus K$ connected, the complex-ellipticity and $\Acal u \equiv 0$ imply that the precise representative of $u$ on $\Omega \setminus K$ is a polynomial $p \in V\otimes \R[x_1,\dots,x_n]$ in the kernel of $\Acal$ (see Sec.~\ref{sec:properties}). Since also $\Leb^n(K) = 0$, it follows that $u \equiv p \in (L^1_\loc \cap C^\infty)(\R^n;V)$ on $\R^n$. In particular, $u$ is continuously differentiable and satisfies $\Acal u(x) = 0$ for all $x \in \R^n$. 
\end{proof}

\begin{proof}[Proof of Proposition \ref{prop:alg}]
	We may without loss of generality assume that $V = \R^n$ and $W= \R^M$. Observe that the complex bilinear form $f$ associated to $\Acal$ is non-singular if and only if $\Acal$ is $\C$-elliptic. To see this, notice that there exists a bijective correspondence between first order complex-elliptic operators and non-singular bilinear maps from $\C \otimes V$ to $\C \otimes W$. Namely, for a first-order operator $\Acal$, we may write
	\begin{equation}\label{eq:bilin}
		f_\Abb(\xi,a) = \Abb(\xi)[a],
	\end{equation}
	and the complex-ellipticity of $\Acal$ tells us that
	\[
	f_\Abb(a,\xi) = 0 \quad \iff \quad a = 0 \; \text{or} \; \xi = 0.
	\]
	Proposition~1.3 in~\cite{larry} then implies $M + 1 \ge N + n$ as desired.	If $N=1$, the gradient operator $D = (\partial_1,\dots,\partial_n)$ attains the bound (the case for $n = 1$ is symmetric).
	To see that the bound is sharp for $N,n \ge 2$ we construct a non-singular bilinear form using Cauchy products: let $p,q \ge 1$ and consider the bilinear form 
	\[
	f(x,y) : \C^{p+1} \times \C^{q+1} \to \C^{p+q+1} : (x,y) \mapsto (z_0,\dots,z_{p+q}),
	\]
	where $x = (x_0,\dots,x_p)$, $y = (y_0,\dots,y_q)$ and
	\[
	z_m \coloneqq \sum_{r+s=m} x_ry_s.
	\]
	By an induction argument it is easy to verify that $f$ defines a complex non-singular bilinear form. Therefore, the partial differential operator $\Fcal: \Crm^\infty(\R^{p+1};\R^{q+1}) \to \Crm^\infty(\R^{p+1};\R^{p+q+1})$ whose symbol satisfies
	\[
	\Fbb(\xi)a = f(\xi,a), \quad (\xi,a) \in \R^{p+1} \times \R^{q+1},
	\]
	is $\C$-elliptic. 	
\end{proof}

    \section{Examples}\label{sct:ex}
    In this section we gather some known operators which satisfy the complex-ellipticity property.

    \begin{example}[Gradient] For $u: \R^n \to V$, the gradient operator
        \[
        D u = (\partial_1 u, \dots, \partial_n u),
        \]
        is a complex-elliptic operator. Indeed,  the symbol associated to $D$ is simply
        \[
        D(\xi)[a] = a \otimes \xi, \qquad a \in V, \ \xi \in \R^n,
        \]
        which has no complex non-trivial zeros.     \end{example}
    
    \begin{example}[Principal derivatives]
        For a scalar map $u$, the `diagonal' of the $k$\textsuperscript{th} order Hessian tensor
        \[
        u \mapsto \mathrm{diag}(D^k u) \coloneqq (\partial^k_1 u,\dots,\partial^k_n u)
        \]
        is complex-elliptic. This follows from the observation that   $\{\xi^k_1,\dots,\xi_n^k\}$ is a family of polynomials in $\C[\mathbf x]$ with no common non-trivial zeros. 
            \end{example}

    \begin{example}[Symmetric gradient, \cites{Suq,Tem_Str}] For vector-valued map $u: \Omega \subset \R^n \to \R^n$ we define its symmetric gradient
        \begin{align*}
        Eu & \coloneqq \frac12 (Du + Du^t) \\
        &  = \frac 12 (\partial_j u^i + \partial_i u^j)_{i,j} \qquad i,j = 1,\dots,n\,,
        \end{align*}
        which takes values on the space $E_2(\R^n)$ of symmetric bilinear forms of $\R^n$. Clearly,
        \[
        E(\xi)[a] = a \odot \xi.
        \]
        It is easy to check that $E$ is complex-elliptic because its nullspace on open connected sets is finite-dimensional. Indeed, it is well-known to be the space affine transformations $Rx + c$ where $R \in \R^n \otimes \R^n$ is a skew-symmetric matrix and $c \in \R^n$. See \cite{kohn1979new-estimates-f} and also \cite{hajlasz_on_appr_diff_bd_96} for more details on the symmetric gradient and its complex-ellipticity.
          \end{example}
    
    More generally, one may consider the symmetrization of the gradient of a symmetric $k$-tensor field:
    
    \begin{example}[Symmetric $k$-tensor field gradients, \cite{van-schaftingen2013limiting-sobole}]
        Let $k \in \N$. For a symmetric $k$-tensor $v \in E_k(\R^n)$, we define the operator with symbol $E^k(\xi)[v] \in E_{k+1}(\R^n)$, where
        \[
        E^k(\xi)v[a_1,...,a_{k+1}] \coloneqq \frac{1}{(k+1)!}\sum_{\sigma \in S_{k+1}} (\xi\cdot a_{\sigma(k+1)})v[a_{\sigma(1)},...,a_{\sigma(k)}].
        \]
        Thus, for a $k$-tensor-field $u : \R^n \to E_k(\R^n)$, we have
        \[
        E^ku = \sym^{k+1}(D u),
        \]
        where $\sym^\ell(v)$ denotes the $\ell$-symmetrization of an $\ell$-tensor $v \in E_\ell(\R^n)$. We recall the argument in~\cite{van-schaftingen2013limiting-sobole}*{Proposition~6.5}, which extends to complex variables: Let $\xi \in \C^n\setminus\{0\}$ and let $v \in \sym_{k}(\C^n)$. Then $E^k(\xi)[v] = 0$ implies 
        $E^k(\xi)v[a,\dots,a] = 0$ for all $a \in \C^n$. This however implies
        \[
        v[a,\dots,a] = 0 \qquad \text{for all $a \in \C^n$ with $a \cdot \xi \neq 0$.}
        \]
        This implies that $v = 0$. 
    \end{example}
    
    \begin{example}[Deviatoric operator, \cite{Resh}]\label{ex:dev} The operator  that considers only the shear part of the symmetric gradient:
        \[
        Lu = Eu - \frac{\Div(u)}{n}I_n 
        \]	
        is elliptic. If $n \ge 3$, then it is also $\C$-elliptic. 
    \end{example}
    \begin{proof} 
        $\Lbb(\xi)a = 0$ if and only if $n(a \odot \xi) = (a \cdot \xi) I_n$. Since $\rank_{\C}(a \odot \xi) \le 2$ for all $\xi,a \in \C^n$, the operator is clearly complex-elliptic whenever $n \ge 3$ (and elliptic for all $n \ge 1$ because its real-eigenvalues $\lambda_1,\lambda_2$ satisfy $\lambda_1 \times \lambda_2 \le 0$).  
        To see that indeed it fails to be complex-elliptic for $n = 2$, we observe that in this case we may re-write $Lu$ as
        \[
        Lu = \curl(u_2,u_1)\frac{(e_1 \otimes e_1 - e_2 \otimes e_2)}2 + \Div(u_2,u_1) e_1 \odot e_2.
        \]
        Therefore, the equation $Lu = 0$ is equivalent to the div-curl equations, which are associated with the symbol $\tilde \Lbb(\xi)[v] = (\xi \cdot v, \xi^\perp \cdot v)$. The latter, however, is not complex-elliptic since $\tilde L(1,\mathrm i)[1,\mathrm i] = 0$.   \end{proof}

    \begin{example}\label{ex:cauchy}
        The operator $\Ccal : \Crm^\infty(\R^n;\R^n) \to \Crm^\infty(\R^n;\R^{N + n-1})$ defined by 
        \[
        \Ccal u = \bigg(\sum_{r + s = m} \partial_r  u^{s}\bigg)_{m}, \qquad m -1 = 1,\dots,n+ N -1\,,
        \]	
        is complex-elliptic (see Proposition~\ref{prop:alg}). 
    \end{example}
\appendix
      \section{A useful factorization for higher-order operators}\label{sec:h} 
    
    The goal of this section is to introduce a factorization for higher-order operators that reduces the complexity of analyzing $k$th-order operators to that of first-order methods. 
    The idea of the construction is to factorize a $k$th order operator $\mathcal A(D)$ as 
    \[
        \Acal(D) = L \circ \mathcal B(D) \circ D^{k-1}\,,
    \]
    where $L$ is a linear map and $\mathcal B(D)$ is a first-order operator with a comparable symbol in the following sense: if $\xi \in \R^n$, then
    \[
       \ker \Bbb(\xi) = \ker \Abb^k (\xi) \otimes V_\xi^{k-1}\,, \qquad  V_\xi^r \coloneqq \spn \{\xi^{\otimes^{r}}\}
    \]
    Up to linear isomorphism, the mapping $\mathcal A(D) \mapsto \Bcal(D)$ is one to one and has the crucial property to retain ellipticity, constant rank, and other relevant conditions attached to the symbol's kernel. 
    
    A simple way to visualize the construction is to recall the factorization of the Laplacian in terms of the divergence operator $\Delta u = \diverg (Du)$. In this example, one may consider the first-order operator 
    \[
        \mathcal B(D)w \coloneqq \diverg w \oplus \curl w\,, \qquad w : \R^n \to \R^n\,,
    \]
    and $L: \R \times \R^{\binom{n}{2}}$ to be the projection on the first coordinate.

 The general construction is recorded in the following result:

    \begin{theorem}[Order reduction]\label{1.1} Let $k \ge 2$ and let $\Acal$ be a $k$th-order homogeneous operator on $\R^n$-variables, from a space $V$ to a space $W$. Then, there exists a first-order operator $\Bcal$ on $\R^n$-variables, from $V \otimes E_{k-1}(V)$ to a vector space $Z$ containing $W$, satisfying the following properties:\vskip8pt 
    \begin{enumerate}[label={(\roman*)}] \setlength{\itemsep}{1em}
    \item Functional factorization: $\Acal$ can be decomposed as 
    \begin{equation*}
        \Acal(D) = P_W \circ \Bcal(D) \circ D^{k-1}
    \end{equation*}
     where $P_W : Z \to W$ denotes the canonical linear projection.

    \item Algebraic stability: If $\Kbb \in \{\R,\C\}$, then
    \begin{equation*}
        \Abb^k(\xi)[v] = P_W \circ \Bbb(\xi) [v \otimes \xi^{\otimes^{k-1}}] \qquad \text{for all} \; (\xi,v) \in \Kbb^n \times (\Kbb \otimes V)\,.
    \end{equation*}
    \item Kernel stability: If $\Kbb \in \{\R,\C\}$, then 
    \[
        \ker_\Kbb \Bbb(\xi) = \ker_\Kbb \Abb^k(\xi)  \otimes \spn_\Kbb \{\xi^{\otimes^{k-1}}\}\qquad \text{for all $\xi \in \Kbb^n\,.$}
    \]
    \end{enumerate}
    \end{theorem}

    \begin{remark}\label{rem:stabil} Our order-reduction decomposition remains invariant under several useful symbolic manipulations. We record the following some important examples with $\Kbb = \{\R,\C\}$:  \vskip1em
    \begin{enumerate}[label={(\alph*)}]\setlength{\itemsep}{1em}
        \item $\Acal$ is $\Kbb$-elliptic if and only if $\Bcal$ is $\Kbb$-elliptic: for every $\xi \in \Kbb^n$ we have
        \[
            \ker_\Kbb \Abb^k(\xi) = \{0\} \quad \Longleftrightarrow \quad \ker_\Kbb \Bbb(\xi)= \{0\}\,.
        \]
        \item More generally, the $\Kbb$-rank of the symbol is invariant, i.e., 
         \[
            \rank_\Kbb (\Abb^k(\xi)) = r \quad \Longleftrightarrow \quad  \rank_\Kbb (\Bbb(\xi)) = r\,.
        \]
        \item $\Acal$ is boundary elliptic in the direction $\nu \in \mathbb S^{n-1}$ if and only if $\Bcal$ is boundary elliptic in the same direction: for every $\xi \in \R^n$ it holds 
        \[
            \ker_\C \Abb^k(\xi + \mathrm i \nu) = \{0\}\quad \Longleftrightarrow \quad  \ker_\C \Bbb(\xi + \mathrm i \nu)= \{0\}\,.
        \]
        \item $\Acal$ is cocanceling  if and only if $\Bcal$ is cocanceling , i.e., 
        \[
            \bigcap_{\xi \in S^{n-1}} \ker \Abb^k(\xi) = \{0\} \quad \Longleftrightarrow \quad \bigcap_{\xi \in S^{n-1}} \ker \Bbb(\xi) = \{0\}\,.
        \]
        \item Other mixing conditions related to the symbol's kernel remain invariant under the order-reduction procedure. In particular, the mixing conditions introduced in~\cites{AR-PAMS,arroyo2020slicing,GAFA}, concerning the fine structure of PDE-constrained measures, also remain invariant. 
         \end{enumerate}
    \end{remark}

   \subsection{Applications to theory of higher-order operators}Theorem~\ref{1.1} enables the analysis of higher-order operators by leveraging insights from the simpler analysis of first-order operators. This order reduction translates to a significant simplification for various problems, as evidenced by the substantial effort required to extend first-order results to the higher-order setting in several works: \vskip1em
   \begin{enumerate}[left=0.3em,itemsep=1em]      
   \item[$\bullet$] The $L^1$ estimates for vector fields under higher-order differential conditions (\cites{VS0,VS1,van-schaftingen2013limiting-sobole}), which Van Schaftingen established as a generalization of the renowned Bourgain--Brezis estimates (\cites{BB1,BB2}).
       \item[$\bullet$] The estimates for higher-order operators established in~\cite{GR17} can be seen to follow directly from the existence of traces for first-order complex-elliptic operators in~\cite{breit2017traces}. 
       \item[$\bullet$] The existence of traces for higher-order operators from~\cite{VS2}*{Theorem~5.2} (see also~\cite{VS2}*{Open problem~5.3} in that reference) follow from the order-reduction and the results in~\cite{breit2017traces}.
   \item[$\bullet$] Theorem~\ref{1.1} provides new insight about the validity of Sobolev inequalities on half-spaces. More precisely, in~\cite{gmeineder2024boundary}*{Theorem 1.1} it was shown that if $\Acal$ is $k$th order operator, then $\Acal$ is elliptic \emph{and} boundary elliptic in the direction $\nu \in S^{n-1}$ if and only if for all $u \in C^\infty_c(\R^n;V)$
   \begin{equation}\label{eq:poincare}
        \|D^{k-1} u\|_{L^\frac{n}{n-1}(H^+_\nu)} \le c \| \Acal u\|_{L^1(H^+_\nu)}\,.
   \end{equation}
   Here, 
    $H_\nu$ is the half-space in $\R^n$ orthogonal to $\nu$. The authors established the Sobolev estimate~\eqref{eq:poincare} for $k>1$, by proving the  \emph{boundary ellipticity} is equivalent with a \emph{trace estimate} (cf.~\cite{gmeineder2024boundary}*{Theorem 1.3})
    \begin{equation}\label{eq:poincare2}
   \begin{aligned}
        \|\partial_\nu^{k-1}  u \|_{L^1(H_\nu)}  \; + \; &  \\ 
        \{\textrm{Besov trace-norms on lower order}& \; \partial_\nu^{k-1 -j} \} 
    \end{aligned}
    \; \le \;  c \|\Acal u\|_{L^1(H_\nu^+)}\,.
    \end{equation}
      Theorem~\ref{1.1} and Remark~\ref{rem:stabil} reveal that    it is possible to give a proof of~\eqref{eq:poincare} 
   without passing through~\eqref{eq:poincare2}. In fact, our reduction argument establishes directly (through~\cite{gmeineder2024boundary}*{Theorem 1.2}) that the trace estimate 
   \[
        \|D^{k-1}  u\|_{L^1(H_\nu)} \le c \|\mathcal Au\|_{L^1(H_\nu^+)}
   \]
   is a \emph{necessary} condition for the boundary ellipticity of $\Acal$ with respect to $\nu$. 
   Establishing the necessity of the lower order Besov norms (shown in~\cite{gmeineder2024boundary}*{Theorem 1.2}) seems to require an additional (and perhaps nontrivial) argument that does not follow directly from our order-reduction analysis.
   \end{enumerate}

  \subsection{Proof of Theorem~\ref{1.1}}   The proof is divided into the following steps: 

    \emph{1. Construction of $\Bcal$.}
        We write $\Acal$ in jet-notation, that is, $\Acal = A[D^ku]$ where $A : V \otimes E_k(\R^n) \to W$ is a linear map. Since $A$ acts linearly on $D^k u$, we can further decompose $\Acal u$ as
        \begin{equation}\label{eq:first_id}
            \Acal u = A[D \circ D^{k-1}u] = \tilde \Acal \circ D^{k-1}u\,,
        \end{equation}
        where $\tilde \Acal$ is a first-order operator from $V \otimes E_{k-1}(\R^n)$ to $W$. Let us now consider the vector space $Z \coloneqq W \oplus (V \otimes^k \R^n)$ and the first-order operator, from $V \otimes E_{k-1}(\R^n)$ to $Z$, defined by 
        \[
            \Bcal  \coloneqq  \tilde \Acal \oplus \Curl_{k-1}.
        \]
        Here, for a positive integer $m$, the operator $\Curl_m$ is the first-order operator from $V \otimes E_m(\R^n)$ to $V \otimes^k \R^n$ defined as (see~\cite{FM99}*{Remark~3.3(iii)})
        \[
        \big(\!\Curl_m w\big)_j = \partial_i w^j_{\beta+ e_\ell} - \partial_\ell w^j_{\beta + e_i} \quad 1 \le j \le M,\, \beta \in \N^n,\, |\beta| = m-1.
        \]
        In  particular, since $\Curl_{k-1} \circ D^{k-1} \equiv 0$, we conclude that
        \[
            P_W \circ \Bcal \circ D^{k-1} = P_W (\tilde \Acal \circ D^{k-1} \oplus 0 ) = \tilde \Acal \circ D^{k-1} = \Acal\,.
        \]
        This proves (i) and also (ii) by taking applying the Fourier transform to the same identity.

        \emph{2. Characterizing the kernel of $\Curl_m$.} We claim that
        \[
        \ker_\Kbb (\Curl_m)(\xi) = \setb{v \otimes \xi^{\otimes^m}}{v \in \Kbb \otimes V, \ \xi \in \Kbb^n}\,.
        \]
        Fix a non-zero direction $\xi \in \Kbb^n$  and a  tensor $F \in \ker_\Kbb(\Curl_m)(\xi)$. Then by definition we obtain
        \[
        \xi_i F^j_{\beta + e_\ell} = \xi_\ell F^j_{\beta + e_i}, \quad
        1 \le j \le M,\, \beta \in \N^n,\, |\beta| = m-1.  
        \]
        Since $\xi$ is non-zero, we may find an index $1 \le i \le n$ such that $\xi_{i} \in \Kbb$ is non-zero. Hence, for this $\xi_i$ we obtain the relations
        \begin{equation}\label{eq:curl}
        F^j_{\beta + e_\ell} = \xi_\ell \bigg(\frac{F^j_{\beta + e_i}}{\xi_{i}}\bigg) \eqqcolon \xi_\ell E^j_{\beta},
        \end{equation}
        for all $\ell \in \{1,\dots,n\}\setminus \{i\}$, $1 \le j \le N$, and all multi-indexes $\beta \in \Nbb^n$ with $|\beta| =m$. It is easy to check that $E^j_\beta$ is indeed independent of the choice of index $i$.
         Furthermore, from the symmetries of $F$ we deduce that 
        $F = E \otimes \xi$ for some $E \in \Kbb \otimes V$. 
        The fact that~\eqref{eq:curl} holds with $\beta = \gamma + e_\ell$ for all multi-indexes $\gamma \in \Nbb^n$ with modulus $|\gamma| = m-2$ yields (again, by exploiting the symmetries of $F$) that $\Curl_{m - 1} E = 0$. 
        Therefore, after a suitable inductive argument over $m$, we finally obtain that $F$ has the form 
        \[
        F = v \otimes \xi^{\otimes^m} \quad \text{for some $v \in \Kbb \otimes V$}.
        \]
        This proves the claim. 

         \emph{3. Establishing the symbolic identity.} We now turn back to the computation of $\ker_\Kbb \Bbb(\xi)$. By the universal property of direct products, we have that that 
        \begin{equation}\label{eq:prod}
        \ker_\Kbb \Bbb(\xi) = \ker_\Kbb \tilde \Abb (\xi) \cap \ker_\Kbb (\Curl_{k-1}) (\xi).
        \end{equation} The second step tells us that we may restrict to elements $F =  v \otimes \xi^{\otimes^{k-1}} $ when computing $\ker_\Kbb \Bbb (\xi)$. With this mind, let $\xi \in \Kbb^n$ be fixed. We get 
        \begin{align*}
            \ker_\Kbb \Bbb(\xi) & = \set{v \otimes \xi^{\otimes^{k-1}}}{v \in \Kbb \otimes V , \, \tilde \Abb(\xi)[v \otimes  \xi^{\otimes^{k-1}}] = 0} \\
            &  = \set{v \otimes \xi^{\otimes^{k-1}}}{v \in \Kbb \otimes V , \, \tilde \Abb(\xi) \circ D^{k-1}(\xi)[v] = 0} \\
            &  = \set{v \otimes \xi^{\otimes^{k-1}}}{v \in \Kbb \otimes V , \,  \Abb^k(\xi)[v] = 0} = \ker_{\Kbb} \Abb^k(\xi) \otimes \spn_\Kbb\{\xi^{\otimes^{k-1}}\}\,.
        \end{align*}
        Here, we have used~\eqref{eq:first_id} in passing to the one but last identity. This proves (iii). This finishes the proof.

    \subsection*{Data Availability Statement}
    Research data/code associated with this article are available in arXiv, under the reference https://doi.org/10.48550/arXiv.1911.08474\,.

\end{document}